\documentclass{article}[a4paper]
\usepackage[utf8]{inputenc}
\usepackage{amsfonts}
\usepackage{enumerate}
\usepackage{amsmath,amsthm,amssymb}
\usepackage[utf8]{inputenc}
\usepackage[margin=1in]{geometry}
\usepackage{graphicx}
\usepackage[english]{babel}
\makeindex
\usepackage{fancyhdr}
\usepackage{nameref} 
\usepackage[perpage]{footmisc}
\usepackage{mathtools}
\usepackage[mathscr]{eucal}
\usepackage{latexsym}
\usepackage{tikz-cd}
\usepackage{titlesec}
\usetikzlibrary{decorations.pathmorphing}
\usepackage{dsfont}
\usepackage{bbm}
\usepackage{etoolbox}
\usepackage{csquotes}

\usepackage{pgf,tikz,pgfplots}
\pgfplotsset{compat=1.15}
\usepackage{mathrsfs}
\usetikzlibrary{arrows}
\usepackage[bbgreekl]{mathbbol}

\usepackage[
backend=biber,
style=numeric,
doi=false,
url=false,
isbn=false,
eprint=false
]{biblatex}
\usepackage{hyperref}
\usepackage[hyphenbreaks]{breakurl}
\usepackage[noabbrev]{cleveref}

\DeclareMathOperator{\Hom}{Hom}
\DeclareMathOperator{\Cont}{Cont}

\DeclareMathOperator{\Sh}{Sh}

\DeclareMathOperator{\Ext}{Ext}

\DeclareMathOperator{\Gal}{Gal}

\newcommand{\N}{\ensuremath{\mathbbm{N}}}
\newcommand{\Q}{\ensuremath{\mathbbm{Q}}}
\newcommand{\R}{\ensuremath{\mathbbm{R}}}
\newcommand{\Z}{\ensuremath{\mathbbm{Z}}}

\newcommand{\CC}{\ensuremath{\mathscr{C}}}
\newcommand{\NN}{\ensuremath{\mathcal{N}}}

\newcommand{\DDD}{\ensuremath{\mathbf{D}}}
\renewcommand{\SS}{\ensuremath{\mathcal{S}}}
\newcommand{\SSS}{\ensuremath{\mathbb{S}}}
\newcommand{\FF}{\ensuremath{\mathcal{F}}}
\newcommand{\GG}{\ensuremath{\mathcal{G}}}
\renewcommand{\H}{\mathrm{H}}

\renewcommand{\O}{\ensuremath{\mathcal{O}}}
\newcommand{\UU}{\ensuremath{\mathcal{U}}}
\newcommand{\HH}{\ensuremath{\mathbbm{H}}}
\newcommand{\GGamma}{\ensuremath{\mathbb{\Gamma}}}

\newcommand{\LH}{\ensuremath{\mathcal{LH}}}
\newcommand{\Set}{\mathbf{Set}}

\newcommand{\Grp}{\mathbf{Grp}}
\newcommand{\Ab}{\mathbf{Ab}}

\newcommand{\Top}{\mathbf{Top}}

\newcommand{\Cond}{\mathbf{Cond}}
\newcommand{\LCA}{\mathbf{LCA}}
\newcommand{\FLCA}{\mathbf{FLCA}}

\newcommand{\Solid}{\mathbf{Solid}}

\makeatletter
\newtheorem*{rep@theorem}{\rep@title}
\newcommand{\newreptheorem}[2]{%
\newenvironment{rep#1}[1]{%
 \def\rep@title{#2 \ref{##1}}%
 \begin{rep@theorem}}%
 {\end{rep@theorem}}}
\makeatother

\theoremstyle{plain}
\newtheorem{thm}{Theorem}[section]
\newreptheorem{thm}{Theorem}
\newtheorem{cor}[thm]{Corollary}
\newreptheorem{cor}{Corollary}
\newtheorem{lem}[thm]{Lemma}
\newreptheorem{lem}{Lemma}
\newtheorem{prop}[thm]{Proposition}
\newreptheorem{prop}{Proposition}

\newtheorem*{thm*}{Theorem}
\newtheorem*{lem*}{Lemma}
\newtheorem*{prop*}{Proposition}
\newtheorem{thmx}{Theorem}

\theoremstyle{definition}
\newtheorem{defn}[thm]{Definition}

\newtheorem{ex}[thm]{Example}
\newtheorem{rmk}[thm]{Remark}
\newtheorem{ntt}[thm]{Notation}

\newtheorem{cns}[thm]{Construction}

 \crefname{cor}{Corollary}{Corollaries}
\crefname{thm}{Theorem}{Theorems}
\crefname{prop}{Proposition}{Propositions}
\crefname{lem}{Lemma}{Lemmas}
\crefname{rmk}{Remark}{Remarks}
\crefname{conv}{Convention}{Conventions}
\crefname{cns}{Construction}{Constructions}
\crefname{ntt}{Notation}{Notation}
\crefname{ex}{Example}{Examples}
\crefname{defn}{Definition}{Definitions}
\crefname{claim}{Claim}{Claims}
\crefname{section}{Section}{Sections}

\newcommand{\pro}[2] {\underset{\leftarrow \, #1}{\lim} \, #2}

\renewcommand{\left}[2] {\prescript{}{#1}{#2}}

\usetikzlibrary{decorations.pathreplacing, calc}

\tikzcdset{
  diagrams={>={Straight Barb[scale=0.5]}}
}
\tikzset{
  altstackar/.style={decorate, decoration={show path construction,
    lineto code={
      \path (\tikzinputsegmentfirst); \pgfgetlastxy{\xstart}{\ystart}
      \path (\tikzinputsegmentlast); \pgfgetlastxy{\xend}{\yend}
      \path ($(0,0)!1.5pt!(\ystart-\yend,\xend-\xstart)$); \pgfgetlastxy{\xperp}{\yperp}
      \foreach \n[evaluate=\n as \k using .5*#1-\n+.5] in {1,...,#1}{
        \ifodd\n{\draw[->, shorten <=2pt, shift={($\k*(\xperp,\yperp)$)}](\xstart,\ystart)--(\xend,\yend);}
        \else{\draw[<-, shorten >=2pt, shift={($\k*(\xperp,\yperp)$)}](\xstart,\ystart)--(\xend,\yend);}\fi
      }
    }
  }}, altstackar/.default={1}
}

\title{Duality for condensed cohomology of the Weil group of a \texorpdfstring{$p$}{p}-adic field}
\author{Marco Artusa}
\date{}

\begin{document}
\maketitle 

\begin{abstract}
We use the theory of Condensed Mathematics to build a condensed cohomology theory for the Weil group of a $p$-adic field. The cohomology groups are proved to be locally compact abelian groups of finite ranks in some special cases. This allows us to enlarge the local Tate Duality to a more general category of non-necessarily discrete coefficients, where it takes the form of a Pontryagin duality between locally compact abelian groups.  
\end{abstract}

\tableofcontents

\newpage

\section{Introduction}
Let $F$ be a finite extension of $\Q_p$ and let $\overline{F}$ be an algebraic closure. Let $G_F$ be the absolute Galois group of $F$. In \cite{TateDuality}, Tate proved that for all finite continuous $G_F$-module $A$ and for all $q$, one has a perfect cup product pairing \begin{equation}\label{tatepairing}
\H^q(G_F,A)\otimes \H^{2-q}(G_F,\Hom(A,\overline{F}^{\times}))\rightarrow \H^2(G_F,\overline{F}^{\times})=\Q/\Z.
\end{equation} However, if $A$ is finitely generated, the result only holds up to profinite completion, which loses important information. There are two promising approaches to overcome this shortcoming of the local Tate duality.

Firstly, Lichtenbaum's influential paper \cite{Licht} suggests that one should replace the Galois group $G_F$ with the Weil group $W_F$. In \cite{Karpuk}, Karpuk studied the cohomology of the latter with discrete coefficients. Secondly, as pointed out by Geisser and Morin in \cite{GeisMor}, one should put a topology on both coefficients and cohomology groups. We try to follow both intuitions simultaneously. 

In this paper we construct topological cohomology groups $\HH^q(B_{\hat{W}_F},-)$ which coincide with Galois cohomology groups $\H^q(G_F,-)$ for finite coefficients and with Weil cohomology groups $\H^q(W_F,-)$ for discrete coefficients. Moreover, we extend the local Tate duality to more general non-discrete coefficients, making it a Pontryagin duality between locally compact abelian groups. In order to do this, we use the theory of Condensed Mathematics \cite{LCM}, which makes it possible to do algebra with topological abelian groups.  A more precise way of saying this is that Condensed Mathematics provides a topos\footnote{We are ignoring set-theoretic issues here. In order to be more precise, we fix $\kappa$ an uncountable strong limit cardinal. Then the topos $\CC$ is the topos of $\kappa$-condensed sets defined in \cite{LCM}.} $\CC$ of condensed sets which contains compactly generated topological spaces as a full subcategory stable by all limits. 

In particular, we consider the prodiscrete topological group $W_F$ as a pro-condensed group, say $\hat{W}_F$. We define its classifying topos and we call it $B_{\hat{W}_F}$. The abelian category $\Ab(B_{\hat{W}_F})$ contains nice enough topological abelian groups with a continuous action of the prodiscrete topological group $W_F$. For all $M\in \Ab(B_{\hat{W}_F})$, or more generally for all objects of the bounded derived $\infty$-category $\DDD^b(B_{\hat{W}_F})$, and for all $q$, we define a condensed abelian group $\HH^q(B_{\hat{W}_F},M)$. Since the category of condensed abelian groups contains the quasi-abelian category $\LCA$ of locally compact abelian groups as a full subcategory, the objects $\HH^q(B_{\hat{W}_F},M)$ may naturally carry a locally compact topology. We define a dualising complex $\R/\Z(1)\in \DDD^b(B_{\hat{W}_F})$. For all $M\in \DDD^b(B_{\hat{W}_F})$, we define a dual module $M^D\coloneqq R\underline{\Hom}(M,\R/\Z(1))\in \DDD^b(B_{\hat{W}_F})$. Our main result can be stated as follows. \begin{repthm}{thm:dualitythm}
Let $M$ be a locally compact abelian group of finite ranks with a continuous action of a finite quotient $G$ of $G_F$. Suppose that $\underline{\Hom}_{\LCA}(\R/\Z,M)$ and $\underline{\Ext}_{\LCA}(\R/\Z,M)$ are finitely generated discrete abelian groups. Then we have a perfect cup-product pairing \begin{equation}\label{introduction:pairing}
\HH^q(B_{\hat{W}_F},M)\otimes \HH^{2-q}(B_{\hat{W}_F},M^D)\rightarrow \HH^2(B_{\hat{W}_F},\R/\Z(1))=\R/\Z
\end{equation} of locally compact abelian groups of finite ranks.
\end{repthm}
We will state and prove this result more generally, for all the objects of a full stable $\infty$-subcategory of $\DDD^b(\FLCA)$, which we denote by $\DDD^{perf}_{\R,\Z}$. 
If $M$ is finite, then we have $M^D=\Hom(M,\overline{F}^{\times})$ (see \cref{dualizingobjectfinite}) and \eqref{introduction:pairing} coincides with the Tate pairing \eqref{tatepairing}. Indeed, the local Tate duality is a key ingredient to prove our result.

\subsection{Outline of this article}
In \cref{section1} we study the cohomology of condensed groups and, whenever the condensed group is represented by a topological group, its relations with other cohomology theories (continuous cohomology, discrete group cohomology). Part of this section is inspired by \cite{Flach}, where Flach uses Grothendieck's ``gros topos" instead of condensed sets. The cohomology of condensed groups is also treated in Emma Brink's Master Thesis. For every condensed group $G$, we define its classifying topos $B_G$ and the condensed cohomology of $G$, i.e.\ functors $\HH^q(B_G,-):\Ab(B_G)\to \Ab(\CC)$ for all $q\in \N$. We apply this to topological groups acting continuously on topological abelian groups, and we obtain a Hochschild-Serre spectral sequence. 
\begin{repprop}{hstopgrps}
Suppose that \[
1\rightarrow H \overset{i}{\rightarrow} G \overset{p}{\rightarrow} Q\rightarrow  1
\] is an exact sequence of topological groups, i.e.\ $Q$ is homeomorphic via $p$ to the coset space $G/H$, and $H$ is homeomorphic to the kernel of $p$ with its subspace topology. Suppose that $\underline{G}\to\underline{Q}$ is an epimorphism in $\CC$. Let $A$ be a topological $G$-module. Then the condensed abelian group $\HH^q(B_{\underline{H}},\underline{A})$ carries a $\underline{Q}$-action for all $q$, and there is an Hochschild-Serre spectral sequence \[
E_2^{p,q}=\HH^p(B_{\underline{Q}},\HH^q(B_{\underline{H}},\underline{A}))\Rightarrow \HH^{p+q}(B_{\underline{G}},\underline{A}).
\]
\end{repprop}
We develop a more general theory for pro-condensed groups, for which \cref{hstopgrps} generalises to a Hochschild-Serre spectral sequence for condensed cohomology of pro-groups (see \cref{prop:proHS}). We also recover a continuity result for the cohomology of strict pro-condensed groups, which generalises the analogous property satisfied by continuous cohomology of profinite groups with discrete coefficients.
\begin{repprop}{prop:cont}
Let $\hat{G}$ be a strict pro-condensed group. Suppose that $G_i$ is compact Hausdorff for all $i$. Let $(A_i,\alpha_{ji})_{i,j\in I}$ be a compatible system of abelian group objects of $(B_{G_i},f_{ij})_{i,j\in I}$ (see \cref{constructionlimittopos}). We set \[
A_{\infty}\coloneqq \underset{\substack{\rightarrow\\ i \in I}}{\lim} \,\pi_i^*A_i\in B_{\hat{G}}.
\] Then the canonical morphism
\[
\underset{\substack{\rightarrow\\ i \in I}}{\lim}\,\HH^q(B_{G_i},A_i)\longrightarrow\HH^q(B_{\hat{G}},A_{\infty})
\] is an isomorphism for any integer $q$.
\end{repprop}

\vspace{0.5 em}
In \cref{section2} we apply the constructions of \cref{section1} to the Weil group of a $p$-adic field. Let $F$ be a $p$-adic field with residue field $k$. Let $W_F$, $I$ and $W_k$ be the Weil group of $F$, the inertia subgroup and the Weil group of $k$ respectively. We have an exact sequence of topological groups $1\to I\to W_F\to W_k\to 1$ which relates $\HH^q(B_I,-),\HH^q(B_{W_F},-)$ and $\HH^q(B_{W_k},-)$ according to \cref{hstopgrps}. Let $L$ be the completion of the maximal unramified extension of $F$, and let $\overline{L}$ be an algebraic closure of $L$ containing an algebraic closure $\overline{F}$ of $F$. Putting the canonical topology on $\overline{L}$, we can see $\overline{L}^{\times}$ as an abelian object of $B_I$. Since $L$ is a C1 field (\cite[\S 3.3, c)]{serreGC}), one would expect $\HH^q(B_I,\overline{L}^{\times})$ to vanish for all $q\ge 1$. However, we prove the following result.
\begin{repprop}{h90failure}
The abelian group $\HH^1(B_I,\overline{L}^{\times})(*)$ is not torsion.
\end{repprop}
In order to overcome this problem, we consider the profinite group $I$ and the prodiscrete group $W_F$ as pro-condensed groups. We obtain classifying topoi $B_{\hat{I}}$ and $B_{\hat{W}_F}$. In this way, the groups $\HH^q(B_{\hat{I}},\overline{L}^{\times})$ and $\HH^q(B_{\hat{W}_F},\overline{L}^{\times})$ behave as in the discrete case (see \cref{vanishing,lcohom}). We define a dualising complex $\R/\Z(1)\in \DDD^b(B_{\hat{W}_F})$ as the fiber of the valuation morphism $\overline{L}^{\times}\to \R$, and we compute $\HH^q(B_{\hat{W}_F},\R/\Z(1))$. 
\begin{repprop}{cohomologyrz1}
The cohomology of $\hat{W}_F$ with coefficients in $\R/\Z(1)$ is given by \[
\HH^q(B_{\hat{W}_F},\R/\Z(1))=\begin{array}{ll}
\underline{\O_F^{\times}} & q=0\\
 \R/\Z & q=1,2\\ 
 0 & q\ge 3,
 \end{array}\] where $\O_F^{\times}$ denotes the units of the ring of integers of $F$, which is a topological abelian group. \end{repprop} For all $M\in \DDD^b(B_{\hat{W}_F})$, we define a dual complex $M^D\coloneqq R\underline{\Hom}(M,\R/\Z(1))\in \DDD^b(B_{\hat{W}_F})$. We also determine the condensed structure on $\HH^q(B_{\hat{W}_F},M)$ and $\HH^q(B_{\hat{W}_F},M^D)$ in some special cases. In particular, we study the case where $M$ is a finitely generated abelian group with a continuous action of a finite quotient of $G_F$ (see \cref{thm:structurem,thm:structuremd}), or a finite-dimensional real vector space with a continuous action of $W_F/U$, where $U$ is an open normal subgroup of $I$ (see \cref{rmk:structurervs}).
\vspace{0.5 em}
 
 Finally, in \cref{section3} we suppose that $M$ is a locally compact abelian group of finite ranks with a continuous action of a finite quotient $G$ of $G_F$. We suppose that $\underline{\Hom}_{\LCA}(\R/\Z,M)$ and $\underline{\Ext}_{\LCA}(\R/\Z,M)$ are finitely generated abelian groups. Note that this is the case, for example, if $M$ is a finitely generated abelian group with the discrete topology or a finite-dimensional real vector space with its Euclidean topology. In this context we prove \cref{thm:dualitythm}.
 
\subsection{Relation to previous work}
As we already remarked, in \cite{TateDuality} Tate proves that for any finite continuous $G_F$-module $M$, the cup-product pairing \eqref{tatepairing} is perfect. This means that the induced maps \[ \psi^q(M):\H^q(G_F,\Hom(M,\overline{F}^{\times}))\rightarrow \H^{2-q}(G_F,M)^*\] and \[
\psi^q(M^D):\H^q(G_F,M)\rightarrow \H^{2-q}(G_F,\Hom(M,\overline{F}^{\times}))^*
\] are isomorphisms for all $q$, where $(-)^*$ denotes the Pontryagin dual. In the attempt of generalising this result to finitely generated continuous $G_F$-modules, Milne proves the following (see \cite[Theorem 2.1]{ADT})
\begin{thmx}\label{thm:milne}
Let $M$ be a finitely generated $G_F$-module, and consider the map \[
\psi^q(M^D): \H^q(G_F,\Hom(M,\overline{F}^{\times}))\rightarrow \H^{2-q}(G_F,M)^*.
\] Then $\psi^q(M)$ is an isomorphism for all $q\ge 1$, and $\psi^0(M^D)$ defines an isomorphism (of profinite groups) \[
\H^0(G_F,\Hom(M,\overline{F}^{\times}))^{\wedge}\rightarrow \H^2(G_F,M)^*.
\]
\end{thmx} 
Hence $\psi^0(M)$ is an isomorphism only up to profinite completion. However, the profinite completion loses a lot of information on the abelian groups. For example, let $M=\Z$ with the trivial action of $\Z$. Then we have \[
\H^0(G_F,\Hom(M,\overline{F}^{\times}))=F^{\times}, \qquad \H^2(G_F,M)^*=G_F^{ab},
\] and $\psi^0(M)$ induces an isomorphism of abelian groups \begin{equation}\label{reciprocity}(F^{\times})^{\wedge}\overset{\sim}{\rightarrow}G_F^{ab},\end{equation} which is the reciprocity isomorphism of Local Class Field Theory. We have $F^{\times}\cong \O_F^{\times}\oplus \Z\cong \Z_p^n\oplus H\oplus \Z$, for some finite abelian group $H$ and some $n\in\N$. Taking the profinite completion, we get \[(F^{\times})^{\wedge}\cong \Z_p^{n+1}\oplus H\oplus \prod_{l\neq p}\Z_l.\] Hence the information coming from $\O_F^{\times}$ and $\Z$ is mixed. In order to resolve this problem, we should replace the Galois group $G_F$ with the Weil group $W_F$, as suggested by Lichtenbaum (see \cite{Licht}).

In \cite{Karpuk}, Karpuk follows this intuition and studies the cohomology of $W_F$ with discrete coefficients. The role of $\overline{F}^{\times}$ is taken by $\overline{L}^{\times}$. The cohomology of $W_F$ with coefficients in $\overline{L}^{\times}$ is given by \[
\H^q(W_F,\overline{L}^{\times})=\begin{array}{ll}
F^{\times} & q=0\\
\Z & q=1 \\
0 & q\ge 2.
\end{array}
\] If $M$ is a finitely generated continuous $W_F$-module, we define $M^D\coloneqq \Hom(M,\overline{L}^{\times})$ and we obtain a cup-product pairing in $\DDD(\Ab)$ \begin{equation}\label{karpukpairing}
R\Gamma(W_F,M)\otimes^L R\Gamma(W_F,M^D)\rightarrow \tau^{\ge 1}R\Gamma(W_F,\overline{L}^{\times})=\Z[-1].
\end{equation}
Karpuk proves the following (see \cite[Theorem 3.3.1]{Karpuk}) \begin{thmx}\label{thm:karpuk}
Suppose that $M$ is a finitely generated continuous $W_F$-module. Then the map \[
\psi(M):R\Gamma(W_F,M^D)\rightarrow R\Hom(R\Gamma(W_F,M),\Z[-1])
\] induced by \eqref{karpukpairing} is a equivalence in $\DDD(\Ab)$. In particular, there are, for all $q$, short exact sequences \begin{equation}\label{karpukses}
0\rightarrow \Ext(\H^{2-q}(W_F,M),\Z)\rightarrow \H^q(W_F,M^D)\rightarrow \Hom(\H^{1-q}(W_F,M^D),\Z)\rightarrow 0.
\end{equation}\end{thmx}
Hence we now have a duality for finitely generated modules which holds in all degrees, without needing the profinite completion. However, since $\Ext(\H^{2-q}(W_F,M),\Z)$ does not vanish in general, \eqref{karpukses} can't express this duality as a duality between the cohomology groups. 

Another aspect to consider is that topology is not taken into account neither by Milne nor by Karpuk. This is the reason why both in Theorem \ref{thm:milne} and in Theorem \ref{thm:karpuk} we don't have perfect pairings, but we only have results about the map $\psi(M)$. We solve this problem by putting a topology both on coefficients and on cohomology groups, as suggested by Geisser and Morin in \cite{GeisMor}. The role played by the discrete $G_F$-module $\overline{F}^{\times}$ for Milne and by the discrete $W_F$-module $\overline{L}^{\times}$ for Karpuk is now played by the complex of condensed $\hat{W}_F$-modules \[
\R/\Z(1)\coloneqq [\overline{L}^{\times}\overset{val}{\rightarrow } \R].
\]
If $M$ is a locally compact abelian groups of finite ranks with a structure of a $\hat{W}_F$-module, we define $M^D\coloneqq R\underline{\Hom}(M,\R/\Z(1))$. In \cref{thm:dualitythm} we show that if $M$ has an action of a finite quotient $G$ of $G_F$ and if $R\underline{\Hom}_{\LCA}(\R/\Z,M)\in \DDD^{perf}(\Z)$, the cup product-pairing \[
\HH^q(B_{\hat{W}_F},M)\otimes \HH^{2-q}(B_{\hat{W}_F},M^D)\rightarrow \HH^2(B_{\hat{W}_F},\R/\Z(1))=\R/\Z
 \] is a perfect pairing of locally compact abelian groups of finite ranks. 

This theorem enlarges the result of Tate to more general non-necessarily discrete coefficients. The ``enlarged" local Tate duality takes the form of a Pontryagin duality between locally compact abelian groups, and in this sense it is richer than Theorems \ref{thm:milne} and \ref{thm:karpuk}. The proof relies on the topological structure of the cohomology groups. In order to determine it, we must suppose that the action of the Weil group is induced by the action of a finite quotient $G$ of $G_F$. Hence, even if Theorem \ref{thm:karpuk} only considers discrete finitely generated coefficients, the hypothesis on the action of $W_F$ is less restrictive in that case. However, if we consider finitely generated abelian groups with a continuous action of $G_F$, as in Theorem \ref{thm:milne}, this ``finitary" property on the action is always satisified.

Finally, as a particular case of \cref{thm:dualitythm}, if $M=\R/\Z$ and $q=1$, the perfect cup-product pairing yields the isomorphism of topological abelian groups \[
F^{\times}\overset{\sim}{\rightarrow} W_F^{ab},
\] which is the reciprocity morphism of Local Class Field Theory ``à la Weil" (compare it with \eqref{reciprocity}), and does not need profinite completion.

\subsection{Set-theoretical conventions and notation}
We say that a category $C$ is small if both $\mathrm{Ob}(C)$ and $\mathrm{Mor}(C)$ are sets. We say that a category is essentially small if it is equivalent to a small category. Let $\kappa$ be an uncountable strong limit cardinal such that $\kappa>\aleph_1$. For example, let $\kappa_0\coloneqq \aleph_1$ and for all $n\in \N_{\ge 1}$, let $\kappa_n\coloneqq 2^{\kappa_{n-1}}$. We set $\kappa\coloneqq \mathrm{sup}_n \kappa_n$. We say that a set $S$ is $\kappa$-small if $|S|<\kappa$.

We denote by $\Top$ the category of topological spaces, and by $\Top^c$ (resp.\ $\Top^{ed}$) the full subcategory of $\Top$ of $\kappa$-small compact Hausdorff topological spaces (resp.\ $\kappa$-small extremally disconnected topological spaces). We observe that $\Top^c$ and $\Top^{ed}$ are essentially small categories. We denote by $\Top^{cg}$ the category of $\kappa$-compactly generated topological spaces, i.e.\ the smallest full subcategory of $\Top$ containing $\Top^c$ and closed under small colimits. 
We denote by $\CC$ the category $\Cond_{\kappa}(\Set)$ of $\kappa$-condensed sets, as defined in \cite{LCM}. Unless stated otherwise, compact Hausdorff (resp.\ extremally disconnected, resp.\ compactly generated, resp.\ condensed) means $\kappa$-small compact Hausdorff (resp.\ $\kappa$-small extremally disconnected, resp.\ $\kappa$-compactly generated, resp.\ $\kappa$-condensed).

We denote by $\LCA$ the quasi-abelian category of locally compact abelian groups, and by $\LCA_{\kappa}$ the quasi-abelian full subcategory of $\kappa$-small objects. We denote by $\FLCA$ the quasi-abelian full subcategory of $\LCA$ of locally compact abelian groups of finite ranks (see \cite{Hoff}). We observe that we have $\FLCA\subset \LCA_{\kappa}\subset \LCA$. The categories $\FLCA$ and $\LCA_{\kappa}$ are essentially small. For a locally compact abelian group $A$, we denote by $A^{\vee}$ its Pontryagin dual, i.e.\ the locally compact abelian group $\Hom_{\LCA}(A,\R/\Z)$ with the compact-open topology. The Pontryagin duality induces an equivalence of categories $\LCA^{op}\cong \LCA$, resp.\ $\LCA_{\kappa}^{op}\cong \LCA_{\kappa}$, resp.\ $\FLCA^{op}\cong \FLCA$. 

Another way of dealing with set-theoretical issues is to follow the conventions of Barwick and Haine (see \cite{BH}). In particular, one can assume the existence of universes. We let $\kappa_0$ be the smallest strong inaccessible cardinal which is uncountable, and $\kappa_1$ the smallest strong inaccessible cardinal with $\kappa_0<\kappa_1$. Then we define the universe $U(\kappa_0)$ (resp.\ $U(\kappa_1)$) as the set of all sets with rank strictly less than $\kappa_0$ (resp.\ $\kappa_1$). The set $U(\kappa_0)$ has rank and cardinality $\kappa_0$, and hence we have $U(\kappa_0)\in U(\kappa_1)$. A mathematical object is $\kappa_0$-small (resp.\ $\kappa_1$-small) if it is equivalent to an object which is in $U(\kappa_0)$ (resp.\ $U(\kappa_1)$). If the readers find this approach more convenient, they can replace $\kappa$-small by $\kappa_0$-small and small by $\kappa_1$-small in the previous discussion. In this case, the categories which in the previous discussion are essentially small, actually become $\kappa_1$-small.

In this article, we use the theory of $\infty$-categories, developed in \cite{HTT}, \cite{HA}, \cite{SAG}. If $\mathcal{A}$ is an abelian category, we denote by $\DDD(\mathcal{A})$ (resp.\ $\DDD^b(\mathcal{A})$, resp.\ $\DDD^+(\mathcal{A})$, resp.\ $\DDD^-(\mathcal{A})$) its derived $\infty$-category (resp.\ bounded derived $\infty$-category, resp.\ bounded-below derived $\infty$-category, resp.\ bounded-above derived $\infty$-category), whose homotopy category is the classical derived category $D(\mathcal{A})$ (resp.\ the bounded derived category $D^b(\mathcal{A})$, resp.\ the bounded-below derived category $D^+(\mathcal{A})$, resp.\ the bounded-above derived category $D^-(\mathcal{A})$).

Finally, we make use of topos theory, which main reference is \cite{SGA4}. In particular, if $T$ is a topos and $X$ is an object of $T$, we denote by $T/X$ the induced topos \cite[IV, \S 5.1]{SGA4}. We call the canonical morphism of topoi $j_X: T/X\rightarrow T$ \cite[IV, \S 5.2]{SGA4} localisation morphism.

Any other unconventional notation will be made clear when it occurs.

\subsection{Acknowledgements}
I would like to deeply thank my advisor Baptiste Morin for suggesting this topic and for his constant support. He introduced me to the world of Condensed Mathematics and to its possible applications to the cohomology of the Weil group. I am thankful to Adrien Morin for the different discussions we had about this subject. I am also grateful to Emma Brink for the useful exchanges about condensed group cohomology, and for her comments on a preliminary version of this article. I also thank Matthias Flach for general discussions about the Weil group and topological group cohomology and for his feedback on this preprint.

\section{Cohomology of condensed groups}\label{section1}
\subsection{Topoi over condensed sets}\label{stronglycompact}
In the following we adapt Morin's definition \cite[Section 8.1]{Mor1} of strongly compact topoi to topoi over condensed sets. The cohomology of a topos is replaced by an enriched cohomology over condensed sets. 
\begin{rmk}\label{rmk:edtopspaces}
We recall some properties of extremally disconnected topological spaces, and their role among condensed sets. \begin{enumerate}[1)]
\item{\cite[Proposition 2.7]{LCM}} A condensed set is a functor $X:(\Top^{ed})^{op}\to \Set$ such that $X(\emptyset)=*$ and $X(S_1\sqcup S_2)=X(S_1)\times X(S_2)$. 
\item For all extremally disconnected $S$, the functor $\Gamma(S,-):\Ab(\CC)\to \Ab, \, A\mapsto A(S)$ commutes with all limits and colimits. This is showed in \cite[Proof of Theorem 2.2]{LCM}.
\item  A morphism of condensed sets $X\to Y$ is an isomorphism if and only if $X(S)\to Y(S)$ is an isomorphism for all $S$ extremally disconnected. This follows from 1) and 2).
\end{enumerate}
\end{rmk}
\begin{defn}
Let $f_T:T\to\CC$ be a topos over $\CC$. We define $\HH^q(T\overset{f_T}{\to}\CC,-)$ as the $q$th derived functor $R^qf_{T*}:\Ab(T)\to \Ab(\CC)$.
\end{defn}
\begin{ntt}Whenever it does not create confusion, we denote $\HH^q(T\overset{f_T}{\to}\CC,-)$ just by $\HH^q(T,-)$. 
\end{ntt}
If we consider the unique morphism of topoi $\CC\to\Set$, whose direct image is the underlying set functor $-(*):\CC\to \Set$, the composition \[
-(*)\circ f_{T*}:T\rightarrow \Set
\] is the global section functor. Since $-(*):\Ab(\CC)\to\Ab$ is exact (see \cref{rmk:edtopspaces}, 2)) the associated Leray spectral sequence degenerates, giving us \[
\HH^q(T,-)(*)=\H^q(T,-).
\]
\begin{rmk}\label{rmk:edsections:enriched}
Let $A\in \Ab(T)$, and let $S$ be an extremally disconnected topological space. Then we have \[ \HH^q(T,A)(S)=\H^q(T/f_T^*S,A_{|f_T^*S}).\] This follows from \cref{rmk:edtopspaces}, 1), since we have\[\H^q(T/f_T^*(S_1\sqcup S_2),A_{|f_T^*(S_1\sqcup S_2)})\cong \H^q(T/f_T^*S_1,A_{|f_T^*S_1})\times \H^q(T/f_T^*S_2,A_{|f_T^*S_2}). \]\end{rmk}
\begin{ntt}\label{underlinedcohomology}
Let $X\in\CC$. Let $f_X:\CC/X\rightarrow \CC$ morphism of topoi. For all $M\in \Ab(\CC)$, we set \[
\H^q(\CC;X,M)\coloneqq \H^q(\CC/X,M_{|X}), \qquad \HH^q(\CC;X,M)\coloneqq \HH^q(\CC/X,M_{|X}).
\] Let $S$ be an extremally disconnected topological space. By \cref{rmk:edsections:enriched} we have \[
\HH^q(\CC;X,M)(S)=\H^q(\CC;X\times S,M).
\]
\end{ntt}
\begin{defn}\begin{enumerate}[i)]
\item A morphism of topoi $f:T_1\to T_2$ is \emph{strongly compact} if for all $q\in\N$ the functor $R^qf_{*}$ commutes with filtered colimits of abelian sheaves. 
\item (\cite[Definition 8.1]{Mor1}) A topos $T$ is strongly compact if the unique morphism of topoi $T\to \Set$ is strongly compact.
\item Let $f_T:T\to \CC$ be a topos over $\CC$. We say that $T$ is \emph{strongly compact over} $\CC$ if $f_T$ is strongly compact. \end{enumerate}\end{defn}
\begin{rmk}\label{rmk:localizediso}
Let $T$ be a topos, $\mathcal{U}\to e_{T}$ a covering of the terminal object and $A\to B$ a morphism in $\Ab(T)$. If $A_{|\mathcal{U}}\to B_{|\mathcal{U}}$ is an isomorphism, then so is $A\to B$. More generally, let $\FF\to \GG$ be a morphism in $\DDD(T)$. If $\FF_{|\mathcal{U}}\to \GG_{|\mathcal{U}}$ is an equivalence, then so is $\FF\to \GG$.
\end{rmk}
\begin{rmk}\label{topoipb}
For every morphism of topoi $f:T_1\to T_2$ and for every object $X$ of $T_2$, the commutative diagram \[
\begin{tikzcd}
T_1/f^*X\arrow[r,"f_{/X}"]\arrow[d,"j_{f*X}"]& T_2/X\arrow[d,"j_X"]\\
T_1\arrow[r,"f"] & T_2
\end{tikzcd}
\] is a pullback. Here $j_X$ and $j_{f^*X}$ are localisation morphisms \cite[IV, 5.2]{SGA4}. Moreover, one has \[
j_X^*\circ R^qf_* = R^q(f_{/X,*})\circ j_{f^*X}^*
\] for all $q$.
\end{rmk}
\begin{lem}\label{sclocalisation}
Let $f:T_1\rightarrow T_2$ be a morphism of topoi. Let $X$ be an object of $T_2$ such that $X\to e_{T_2}$ is a covering of the terminal object. If $f_{/X}$ is strongly compact, then so is $f$.
\end{lem}
\begin{proof}
Combine \cref{rmk:localizediso,topoipb}.
\end{proof}
\begin{lem}\label{compositionscmorphisms}
Strongly compact morphisms of topoi are stable by composition.
\end{lem}
\begin{proof}
Let $f:T_1\to T_2$ and $g:T_2\to T_3$ be strongly compact morphisms of topoi. Let $h\coloneqq g\circ f:T_1\to T_3$. For all $A\in \Ab(T_1)$, we have a spectral sequence \[E_2^{p,q}=R^pg_*(R^qf_*A)\implies R^{p+q}h_* A.\] Since $R^nf_*$ and $R^ng_*$ commute with filtered colimits for all $n$, then so does $R^nh_*$.
\end{proof}
\begin{rmk} The topos $\CC$ of condensed sets is strongly compact. Indeed, the functor $-(*):\Ab(\CC)\to\Ab$ is exact and commutes with filtered colimits (see \cref{rmk:edtopspaces}, 2)). Consequently, if a topos is strongly compact over $\CC$, then it is strongly compact. 
\end{rmk}

A cofiltered limit of strongly compact topoi along strongly compact transition maps is strongly compact \cite[VI, Corollary 8.7.7]{SGA4}. In fact the same can be proved for strongly compact topoi over $\CC$. 
\begin{cns}\label{constructionlimittopos} Let $(T_i,f_{ji})_{i,j\in I}$ be a filtered projective system of topoi over $\CC$, where the maps $f_{ji}:T_j\to T_i$ are the transition maps. Then we set \[
T_{\infty}\coloneqq \underset{\substack{\leftarrow\\ i}}{\lim}\, T_i,
\] where the cofiltered limit of topoi is computed in the 2-category of topoi. For all $i$, we have a morphism $\pi_i:T_{\infty}\to T_i$. The topos $T_{\infty}$ is automatically a topos over $\CC$, and $\pi_i$ is a morphism of topoi over $\CC$ for all $i$.

Let $A_i\in \Ab(T_i)$ for all $i\in I$. Let $(\alpha_{ij}:f_{ji}^*A_i\to A_j)_{i,j\in I}$ be a family of morphisms such that \[\alpha_{ik}=\alpha_{jk}\circ f_{kj}^*(\alpha_{ij}):f_{ki}^*A_i=f_{kj}^*f_{ji}^*A_i\rightarrow f_{kj}^*A_j\rightarrow A_k.
\] 
The morphisms $(\pi_j^*(\alpha_{ij}):\pi_j^*f_{ji}^*A_i=\pi_i^*A_i\rightarrow \pi_j^*A_j)_{i,j\in I}$ yield a filtered inductive system of abelian objects $(\pi_i^*A_i)_{i\in I}$ in $T_{\infty}$. 
Then we set \[
A_{\infty}\coloneqq \underset{\substack{\rightarrow\\ i\in I}}{\lim}\, \pi_i^*A_i.
\] \end{cns}

\begin{lem}\label{lem:2}
Let $(T_i,f_{ji})_{i,j\in I}$ and $T_{\infty}$ be defined as in \cref{constructionlimittopos}. Suppose that $T_i$ is strongly compact over $\CC$ for all $i$, and that the transition maps $f_{ji}:T_j\to T_i$ are strongly compact. Then $T_{\infty}$ is strongly compact over $\CC$.
\end{lem}
\begin{proof}
By \cite[VI, Corollary 8.7.6]{SGA4}, the morphism $\pi_i:T_{\infty}\to T_i$ is strongly compact for all $i$. Since the morphism $T_{\infty}\to \CC$ is the composition of two strongly compact morphisms $\pi_i:T_{\infty}\to T_i$ and $f_{T_i}:T_i\to \CC$, the result follows from \cref{compositionscmorphisms}.
 \end{proof}

\begin{lem}\label{lem:3}
Let $(T_i,f_{ij})_{i,j\in I}$, $(A_i,\alpha_{ij})_{i,j\in I}$, $T_{\infty}$ and $A_{\infty}$ be defined as in \cref{constructionlimittopos}. Suppose that $T_i$ is strongly compact over $\CC$ for all $i$, and that the transition maps $f_{ji}:T_j\to T_i$ are strongly compact. Then the canonical morphism \[
\lim_{\substack{\rightarrow \\ i\in I}} \, \HH^q(T_i,A_i)\rightarrow \HH^q(T_{\infty},A_{\infty})
\] is an isomorphism for any integer $q$.
\end{lem}
\begin{proof}
By \cref{lem:2} and \cite[VI, Corollary 8.7.5]{SGA4} one has \[
\HH^n(T_{\infty},A_{\infty})\cong \underset{\substack{\rightarrow\\ i}}{\lim}\, \HH^q(T_{\infty},\pi_i^*A_i)\cong \underset{\substack{\rightarrow\\ i}}{\lim}\,(\underset{\substack{\rightarrow\\ j\to i}}{\lim}\, \HH^q(T_j,f_{ji}^*A_i)).
\]  We conclude as in the proof of \cite[Lemma 8.2]{Mor1}.
\end{proof}

\subsection{The classifying topos of a condensed group}\label{bg}
\begin{defn} Let $G$ be a condensed group. We define its classifying topos $B_G$ as the category of objects $X$ of $\CC$ with a $G$-action $G\times X\to X$. Morphisms in $B_G$ are those morphisms in $\CC$ which are $G$-equivariant.\end{defn} 
\begin{rmk}The fact that $B_G$ is a topos follows from \cite[IV, \S 2.4]{SGA4}. \end{rmk}
\begin{prop}\label{prop:bgreplete}
Let $G$ be a condensed group. Then its classifying topos $B_G$ is replete.
\end{prop}
\begin{proof}
One has a localization morphism $j_{EG}:B_{G}/EG\rightarrow B_G$. By \cite[Lemma 7, i)]{Flach} we have $B_{G}/EG=\CC$, which is replete by \cite[Proposition 3.2.3]{proetale}. Since $j_{EG}^*$ reflects epimorphisms and commutes with projective limits, the result follows.
\end{proof}
The construction of the classifying topos is functorial. For every morphism of condensed groups $g:G_1\to G_2$ we get a morphism of topoi \[
Bg:B_{G_1}\rightarrow B_{G_2}.
\] The pullback functor $Bg^*:B_{G_2}\to B_{G_1}$ sends $X\in B_{G_2}$ to itself with the action of $G_1$ induced by $g$. Moreover, if $g:G_1\to G_2$ and $g':G_2\to G_3$, then $B(g'\circ g)=Bg'\circ Bg$.
 \begin{ntt}We denote $Bg$ simply by $g$. \end{ntt}
 Let $G$ be a condensed group. The unique map $G\to \{*\}$ induces $f_G:B_G\to B_{\{*\}}=\CC$. We obtain functors \[
\HH^q(B_G,-)\coloneqq R^qf_{G*}(-):\Ab(B_G)\rightarrow \Ab(\CC),
\] as in \cref{stronglycompact}.
\begin{lem}\label{cartanleray}
Let $S$ be an extremally disconnected topological space and let $A\in \Ab(B_G)$. There is a Cartan-Leray spectral sequence \[\label{eqn:cartanleray}
   E_2^{p,q}= \H^p(\H^q(\CC;G^{\bullet}\times S, A))\implies \HH^{p+q}(B_G,A)(S),
\] which is functorial in $S$.
\end{lem}
\begin{proof}
The morphism $G\times S \rightarrow S$ is an epimorphism in $\CC$. Consequently, $EG\times f_G^*S\to f_G^*S$ is an epimorphism in $B_G(\CC)$. We get a simplicial object \[\begin{tikzcd}
    \dotsb\arrow[r, altstackar=7] & S_2\arrow[r, altstackar=5] & S_1\arrow[r, altstackar=3] & S_0 
    \end{tikzcd}\]
 where $S_0=EG\times f_G^*S$ with projection on the second component towards $f_G^*S$, and \[
S_n=S_0\times_{f_G^*S}\times \dots \times_{f_G^*S} S_0,
\] the fiber product of $n+1$ copies of $S_0$, which is isomorphic over $f_G^*S$ to $EG\times f_G^*( G^n \times S)$. We have the Cartan-Leray spectral sequence \[
E_2^{p,q}=\H^p(\H^q(B_G;EG\times f_G^*(G^{\bullet}\times S), A))\implies \H^{p+q}(B_G;f_G^*S,A)
\] By \cite[Lemma 7, i)]{Flach}, with $X=G^n\times S$, we have
\[
\begin{split}
\H^q(B_G;EG\times f_G^*(G^{n}\times S), A)&=\H^q(B_G/EG\times f_G^*(G^n\times S), A|_{EG\times f_G^*(G^n\times S)})\\ &\cong \H^q(\CC/G^{n}\times S, A|_{G^n\times S})=\H^q(\CC;G^n\times S, A).
\end{split}
\] The result follows from \cref{rmk:edsections:enriched}.
\end{proof}
\begin{cor}\label{cartanleray2}
Let $G$ be a condensed group and let $A\in \Ab(B_G)$. There is a Cartan-Leray spectral sequence \[
E_2^{p,q}=\H^p(\HH^q(\CC;G^{\bullet},A))\implies \HH^{p+q}(B_G,A),
\] where $\HH^q(\CC;G^{\bullet},A)$ is defined in \cref{underlinedcohomology}.
\end{cor}
\begin{proof}
By \cref{rmk:edtopspaces}, 2) and 3), it is enough to check that for all $S$ we have a spectral sequence \[
E_2^{p,q}=\H^p(\HH^q(\CC;G^{\bullet},A))(S)\implies \HH^{p+q}(B_G,A)(S)
\] which is functorial in $S$. This is \cref{cartanleray}.
\end{proof}
\begin{defn}\label{condensedgroupcohomology}
Let $G$ be a topological group, and let $A$ be a topological abelian group with a continuous action of $G$. We define the condensed cohomology of $G$ with coefficients in $A$ as $\HH^q(B_{\underline{G}},\underline{A})$.
\end{defn}
\begin{ntt}
If $X$ is a compact Hausdorff topological space/group/abelian group, we denote the condensed set/group/abelian group $\underline{X}$ just by $X$.
\end{ntt}
\begin{prop}\label{bgstronglycompact}
If $G$ is a compact Hausdorff topological group, $B_G$ is strongly compact over $\CC$.
\end{prop}
\begin{proof}
By \cref{cartanleray2} we have a spectral sequence \[
E_2^{p,q}=\H^p(\HH^q(\CC;G^{\bullet},-)) \implies \HH^{p+q}(B_G,-).
\] Since filtered colimits are exact in $\Ab(\CC)$ (\cite[Theorem 2.2]{LCM}), it is enough to show that the topos $\CC/G^n$ is strongly compact over $\CC$ for all $n$. Since $G^n$ is compact Hausdorff for all $n$, this follows from \cite[Proposition 4.12]{LCM}. \end{proof}
\begin{prop}[Hochschild-Serre spectral sequence for topological groups]\label{hstopgrps}
Suppose that \[
1\rightarrow H \overset{i}{\rightarrow} G \overset{p}{\rightarrow} Q\rightarrow  1
\] is an exact sequence of topological groups, i.e.\ $Q$ is homeomorphic via $p$ to the coset space $G/H$, and $H$ is homeomorphic to the kernel of $p$ with its subspace topology. Suppose that $\underline{G}\to\underline{Q}$ is an epimorphism in $\CC$. Let $A$ be a topological $G$-module. Then the condensed abelian group $\HH^q(B_{\underline{H}},i^*\underline{A})$ carries a $\underline{Q}$-action for all $q$, and there is an Hochschild-Serre spectral sequence \[
E_2^{p,q}=\HH^p(B_{\underline{Q}},\HH^q(B_{\underline{H}},i^*\underline{A}))\Rightarrow \HH^{p+q}(B_{\underline{G}},\underline{A}).
\]
\end{prop}
\begin{proof}
This is a special case of \cref{prop:proHS}.
\end{proof}

\subsubsection{Comparison with continuous cohomology}\label{comparisoncontcohom}
 In \cref{cartanleray2} we can recover ``continuous cochains".
\begin{prop}\label{compactopen}
Let $G$ and $A$ be as in \cref{condensedgroupcohomology}. Suppose that $G^n$ is compactly generated for all $n$. Then for all $n\in\N$ the condensed abelian group $\HH^0(\CC;\underline{G^n},\underline{A})$ is represented by $\Cont(G^n,A)$ with the compact-open topology.
\end{prop}
\begin{proof}
For all $S$ extremally disconnected we have
\[
\begin{split}
\HH^0(\CC;\underline{G^n},\underline{A})(S)&=\Hom_{\CC}(\underline{G^n\times S},\underline{A})\cong\Cont(G^n\times S,A)\\&\cong\Cont(S,\Cont(G^n,A))\cong\underline{\Cont(G^n,A)}(S),\end{split}
\] where $\Cont(G^n,A)$ is given the compact-open topology. The second isomorphism is \cite[Proposition 1.7]{LCM}. The result follows from \cref{rmk:edtopspaces}, 3). \end{proof}

\begin{cor}\label{lowercohomologygroups}
Let $G$ and $A$ be as in \cref{condensedgroupcohomology}. Suppose that $G^n$ is compactly generated for all $n$. Then the following holds. \begin{enumerate}[1)]
\item The condensed abelian group $\HH^0(B_{\underline{G}},\underline{A})$ is represented by $A^G$, with the subspace topology induced by $A$.
\item Suppose that $A$ has the trivial action. Then $\HH^1(B_{\underline{G}},\underline{A})$ is represented by $\Hom^{cont}(G,A)$, with the subspace topology induced by the compact-open topology on $\Cont(G,A)$.
\item \label{vanishingcond} Suppose that $\HH^q(\CC;\underline{G^n},\underline{A})=0$ for all $q>0$ and for all $n$. Then $\HH^q(B_{\underline{G}},\underline{A})$ is computed by the complex \[
\underline{A}\rightarrow \underline{\Cont(G,A)}\rightarrow \underline{\Cont(G^2,A)}\rightarrow \dots,
\] where the mapping spaces are given the compact-open topology. In particular, the underlying abelian group $\H^q(B_{\underline{G}},\underline{A})=\HH^q(B_{\underline{G}},\underline{A})(*)$ coincides with continuous group cohomology.
\end{enumerate}
\end{cor}
\begin{proof}
Let $d^0:A\to \Cont(G,A)$ and $d^1:\Cont(G,A)\to\Cont(G^2,A)$ be the continuous morphisms of topological abelian groups defined by \[
d^0(a)\coloneqq(g\mapsto ga-a), \qquad d^1(f)\coloneqq ((g_1,g_2)\mapsto g_1f(g_2)+f(g_1)-f(g_1g_2)).
\] By \cref{cartanleray2,compactopen} we have \[
\HH^0(B_{\underline{G}},\underline{A})=\H^0(\underline{A}\overset{\underline{d^0}}{\longrightarrow} \HH^0(\CC;\underline{G},\underline{A})\rightarrow\dots)=\mathrm{ker}(\underline{d^0}).
\] This proves 1). 

By \cref{cartanleray2} we have an exact sequence of condensed abelian groups \[
0\rightarrow \H^1(\HH^0(\CC;\underline{G}^{\bullet}))\rightarrow \HH^1(B_{\underline{G}},\underline{A})\rightarrow \H^0(\HH^1(\CC;\underline{G}^{\bullet},\underline{A})).
\] Since $\HH^1(\CC;*,\underline{A})=0$, we have $\H^0(\HH^1(\CC;\underline{G}^{\bullet},\underline{A}))=0$. By \cref{compactopen}, we get \[
\HH^1(B_{\underline{G}},\underline{A})=\H^1(\underline{A}\overset{\underline{d^0}}{\longrightarrow} \underline{\Cont(G,A)}\overset{\underline{d^1}}{\longrightarrow} \underline{\Cont(G^2,A)}).
\] Since the action on $A$ is trivial, $d^0$ is the zero morphism and $\mathrm{ker}(d^1)=\Hom^{cont}(G,A)$. This proves 2).

For 3), we just combine \cref{cartanleray2} with \cref{compactopen}.
\end{proof}
\begin{ex}\label{finitegrouptop}
If $G$ is a finite group, the condition in 3) is satisfied for any condensed abelian group $A$. In particular, the following holds:\begin{enumerate}[(i)]
\item if $A$ is represented by a discrete abelian group, then so is $\HH^q(B_G,A)$. Indeed, the functor $\Ab\to \Ab(\CC)$ is exact, hence the quotient of a discrete condensed abelian group by a discrete subgroup is discrete as well.
\item if $A$ is represented by a compact Hausdorff topological abelian group, then so is $\HH^q(B_G,A)$. Indeed, all the terms of the complex are compact Hausdorff, and the morphisms are closed. In addition to this, the functor $\Top\to\CC$ sends surjections between compact Hausdorff topological spaces to epimorphisms.
\end{enumerate}
\end{ex}
\begin{ex}
If $G$ is a profinite group, the condition in 3) is satisfied for any discrete abelian group $A$, and more generally if $A$ is a topological abelian group such that $\underline{A}$ is solid (\cite[Theorem 5.4, Corollary 6.1]{LCM}).
\end{ex}
\begin{ex}\label{ex:vectorspaces}
If $G$ is compact Hausdorff, the condition in 3) is satisfied for any Banach real vector space (\cite[Theorem 3.3]{LCM}).
\end{ex}

\begin{prop}\label{banachcoefficients}
Let $V$ be a Banach real vector space with a continuous action of a compact Hausdorff topological group $G$. Then we have \[
\HH^q(B_G,\underline{V})=0\qquad \forall q>0.
\]
\end{prop}
\begin{proof}
We just combine \cref{ex:vectorspaces} and \cite[IX, Prop 1.12]{Borel}. The exact sequence of Banach spaces is also exact as sequence of condensed abelian groups by the following lemma. \end{proof}

\begin{lem}\label{frechetcomplex}
A complex of Frechet (resp.\ Banach) spaces is exact as a complex of condensed $\R$-modules if and only if it is exact on the underlying vector spaces.
\end{lem}
\begin{proof}
This is a consequence of the Baire category theorem, see \cite[Lemma 11.2]{LCG}. 
\end{proof}

\subsubsection{Discrete coefficients}\label{discretecoefficients}
\begin{rmk}\label{respectscoproducts}
The functor $\underline{(-)}:\Top\rightarrow \CC$ respects coproducts. Indeed, let $\{X_i\}_{i\in I}$ be a family of topological spaces. For all $i$, the inclusion $X_i\hookrightarrow \sqcup_{i\in I} X_i$ is a closed and open immersion. Hence, if $S$ a profinite set and $g$ is a continuous map $S\to \sqcup_{i\in I} X_i$, there exists a partition $S=S_1\sqcup \dots \sqcup S_n$, with $S_i$ profinite, such that $g_{|S_j}$ factors through some $X_{i_j}$.
\end{rmk}
We call $f$ the unique morphism of topoi $\CC\to \Set$. In particular, we have $f_*=-(*)$.
\begin{rmk}
Let $X$ be a set, and let $X^{\delta}$ be the topological space defined by $X$ with the discrete topology. By \cref{respectscoproducts}, $f^*X$ is naturally isomorphic to $\underline{X^{\delta}}$ in $\CC$. \end{rmk}
\begin{defn}
A condensed set $T$ is \emph{discrete} if the natural map $\underline{T(*)^{\delta}}\to T$, or equivalently $f^*f_*T\to T$, is an isomorphism.
\end{defn}
\begin{lem}\label{discretecriterion}
A condensed set $T$ is discrete if and only if for every extremally disconnected $S=\pro{i}{S_i}$\[
\lim_{\substack{\rightarrow\\ i}} T(S_i)\cong T(S).
\]
\end{lem}
\begin{proof}
By \cref{rmk:edtopspaces}, 3), $T$ is discrete if and only if the morphism $\underline{T(*)^{\delta}}(S)\to T(S)$ is an isomorphism for all $S$ extremally disconnected.
 For all $S=\pro{i}{S_i}$ extremally disconnected we have \[
\begin{split} \lim_{\substack{\rightarrow\\i}} T(S_i)&=\lim_{\substack{\rightarrow\\i}} \prod_{s\in S_i} T(*)=\lim_{\substack{\rightarrow\\i}} \mathrm{Maps}(S_i,T(*))\\&=\lim_{\substack{\rightarrow\\i}} \Cont(S_i,T(*)^{\delta})=\Cont(S,T(*)^{\delta})\\&=\underline{T(*)^{\delta}}(S), \end{split}
\] where the second-to-last equality is a computation of $0$th cohomology with discrete coefficients (\cite[Chapter X, Theorem 3.1]{foundations}). The result follows.
\end{proof}
\begin{prop}\label{discretemonoepi} Let $t:T'\to T$ be a morphism of condensed sets.
\begin{enumerate}[i)]
\item Suppose that $t$ is a monomorphism.  If $T$ is discrete, then so is $T'$. 
\item Suppose that $t$ is an epimorphism. If $T'$ is discrete, then so is $T$.
\end{enumerate}
\end{prop}
\begin{proof}
We first prove i). Let $S=\lim_j S_j$ be an extremally disconnected topological space. We have the commutative diagram \begin{equation}\label{diagdiag}
\begin{tikzcd}
\underset{\rightarrow\, j}{\lim}\, T'(S_j)\ar[r]\ar[d] & \underset{\rightarrow\, j}{\lim}\, T(S_j)\arrow[d]\\
T'(S)\arrow[r]&T(S),
\end{tikzcd}
\end{equation} where the map on the right is an isomorphism, and all maps are injective. If the map on the left is surjective, we conclude by \cref{discretecriterion}. Take $a\in T'(S)$. By the diagram above $a\in T(S_i)$ for some $i$. We need to show that $a\in T'(S_i)$. This follows from the sheaf condition on $T'$, namely $T'(S_i)$ is the equaliser of $T'(p_1),T'(p_2):T'(S)\to T'(S\times_{S_i} S)$. 

We prove ii). Let $S=\lim_j S_j$ be an extremally disconnected topological space. The diagram \eqref{diagdiag} is such the map on the left is an isomorphism, and all maps are surjective. Moreover, the map \[\underset{\substack{\rightarrow\\ j}}{\lim}\, T(S_j)\rightarrow T(S)\] is injective, being a filtered colimit of injective maps. We conclude by \cref{discretecriterion}.
\end{proof}
\begin{prop}\label{discreteextensions}
Let $0\rightarrow A'\rightarrow A\rightarrow A''\rightarrow 0$ be an exact sequence of condensed abelian groups. Then $A'$ and $A''$ are both discrete if and only if $A$ is discrete.
\end{prop}
\begin{proof}
If $A$ is discrete, then so are $A'$ and $A''$ by \cref{discretemonoepi}. Conversely, let $S=\lim_j S_j$ be an extremally disconnected topological space. The result follows from \cref{discretecriterion} and the Snake Lemma applied to \[\begin{tikzcd}
0\arrow[r] & \underset{\substack{\rightarrow\\j}}{\lim} \, A'(S_j)\arrow[r]\arrow[d] & \underset{\substack{\rightarrow\\j}}{\lim}\, A(S_j)\arrow[r]\arrow[d] & \underset{\substack{\rightarrow\\j}}{\lim}\, A''(S_j)\arrow[r]\arrow[d]&0 \\
0\arrow[r]& A'(S)\arrow[r]& A(S)\arrow[r] & A''(S)\arrow[r] & 0.
\end{tikzcd}\] 
\end{proof}
\begin{cor}\label{discretespectralsequence}
Let $E_2^{p,q}\implies E^{p+q}$ be a spectral sequence in $\Ab(\CC)$, with $E_2^{p,q}=0$ for $p<0$ or $q<0$. Suppose that $E_2^{p,q}$ is discrete for all $p,q$. Then $E^n$ is discrete for all $n$.
\end{cor}
\begin{proof}
By \cref{discreteextensions}, the subcategory $\Ab\subset \Cond(\Ab)$ is stable by extensions. Hence it is enough to show that for every complex of condensed abelian groups $C^{-1}\overset{d^{-1}}{\rightarrow}C^0\overset{d^0}{\rightarrow} C^1$ such that all the terms are discrete, the cohomology group $\mathrm{ker}(d^0)/\mathrm{im}(d^{-1})$ is discrete as well. This follows from \cref{discreteextensions}.
\end{proof}
 \begin{prop}\label{prop:discretecohom}
Let $G$ be a compact Hausdorff topological group acting continuously on a discrete topological abelian group $A$. Then $
\HH^q(B_{G},\underline{A})$ is discrete for all $q$.
\end{prop}
\begin{proof}
By \cref{cartanleray2,discretespectralsequence}, it is enough to show that $\H^p(\HH^q(G^{\bullet},A))$ is discrete for all $p,q$. Let $S=\lim_i{S_i}$ be an extremally disconnected topological space.
By \cite[Chapter X, Theorem 3.1]{foundations} and \cite[Theorem 3.2]{LCM}, we have \[
\H^q(\CC;G^n\times S,\underline{A})\cong \lim_{\substack{\rightarrow \\ i}} \H^q(\CC;G^n\times S_i,\underline{A}),
\] functorially in $n$. Thus we have \[
\H^p(\H^q(\CC;G^{\bullet}\times S,\underline{A}))\cong \H^p(\lim_{\substack{\rightarrow \\ i}} \H^q(\CC;G^{\bullet}\times S_i,\underline{A})).
\] Since filtered colimits are exact in $\Ab$, we have \[
\lim_{\substack{\rightarrow \\ i }} \H^p(\H^q(\CC;G^{\bullet}\times S_i,\underline{A}))\cong \H^p(\lim_{\substack{\rightarrow\\ i}} \H^q(\CC;G^{\bullet}\times S_i,\underline{A}))\cong \H^p(\H^q(\CC;G^{\bullet}\times S,\underline{A})).
\] The result follows from \cref{discretecriterion}.
\end{proof}
\begin{rmk}
For profinite groups, the same result is shown in \cite[Lemma 2.5]{SolidGC}.
\end{rmk}

\subsubsection{Cohomology of discrete groups}\label{cohomologydiscretegroup}
\begin{prop}\label{prop:discretegroupsections}
Let $G$ be a discrete group and let $A$ be a condensed $\underline{G}$-module. Then, for all $q$ and for all $S$ extremally disconnected, we have \[
\HH^q(B_{\underline{G}},A)(S)\cong \H^q(B_G(\Set),A(S)),
\] where the right-hand side is the classical group cohomology of $G$ with coefficients in $A(S)$.
\end{prop}
\begin{proof}
By \cref{topoipb}, the commutative diagram of topoi \[
\begin{tikzcd}
\CC/S \ar[r]\ar[d]&\Set\ar[d,"p"]\\
B_{\underline{G}}/S\ar[r,"f"]&B_G(\Set)
\end{tikzcd}
\] is a pull-back square and $p$ is a localisation morphism. Indeed, we have $\Set=B_G(\Set)/EG$ and $\CC/S=(B_{\underline{G}}/S)/f^*EG$. Thus we get \[\begin{split}
p^*Rf_*(A_{|S})&\cong R\Gamma(\CC/S,A_{|S})\\
&\cong R\Hom_{\Ab(\CC)}(\Z[S],A)\\
&\cong \Hom_{\Ab(\CC)}(\Z[S],A)\\
&\cong A(S),
\end{split}
\] where we use the fact that $\Hom_{\Ab(\CC)}(\Z[S],-)$ is exact if $S$ extremally disconnected \cref{rmk:edtopspaces}, 2). Hence $Rf_*(A_{|S})$ is concentrated in degree $0$ and we have
\[
\HH^q(B_{\underline{G}},A)(S)=R^q\Gamma(B_{\underline{G}}/S,A_{|S})\cong R^q\Gamma(B_G(\Set),f_*(A_{|S}))\cong\H^q(B_G(\Set),A(S)).
\]
\end{proof}

\begin{ex}[Cohomology of $\Z$]\label{cohomologyz}
Let $G=\Z=<\Phi>$ with the discrete topology. Then $
\Z[\Z]\overset{\cdot(1-\Phi)}{\longrightarrow}\Z[\Z]\overset{ev}{\longrightarrow} \Z$ is a projective resolution of $\Z$. Let $A\in \Ab(B_{\Z})$. For all extremally disconnected $S$, let $\varphi_S:A(S)\to A(S)$ be the automorphism of $A(S)$ induced by the action of $\Phi$. By \cref{prop:discretegroupsections}, we have \[
\HH^0(B_{\Z},A)(S)\cong \mathrm{ker}(1-\varphi_S), \qquad \HH^1(B_{\Z},A)(S)=\mathrm{coker}(1-\varphi_S), \qquad \HH^q(B_{\Z},A)=0, \, \forall q\ge 2.
\] Let $\varphi:A\to A$ be the automorphism of $A$ defined by $\varphi(S)\coloneqq \varphi_S$.  Then we have \[
\HH^0(B_{\Z},A)=\mathrm{ker}(1-\varphi)\eqqcolon A^{\Z}, \qquad \HH^0(B_{\Z},A)=\mathrm{coker}(1-\varphi)\eqqcolon A_{\Z}, \qquad \HH^q(B_{\Z},A)=0\,\forall q\ge 2.
\]
If $A$ is a topological abelian group with a continuous action of $\Z$, the morphism $\varphi$ is represented by the automorphism of topological abelian groups $A\to A$, mapping $a$ to $\Phi\cdot a$. 
Consequently, $\underline{A}^{\Z}$ is represented by the topological abelian group $A^{\Z}$ (in agreement with \cref{lowercohomologygroups}, 1)). Moreover, if $A$ is discrete, $\underline{A}_{\Z}$ is represented by the discrete abelian group $A/(1-\varphi)A$. If $A$ is compact Hausdorff, the condensed abelian group $\underline{A}_{\Z}$ is represented by the compact Hausdorff abelian group $A/(1-\varphi)A$ with the quotient topology. Indeed, the morphism $1-\varphi$ is closed.
\end{ex}

\subsection{The classifying topos of a pro-condensed group}\label{progroups}
We would like to generalize a well known fact from the cohomology of profinite groups with discrete coefficients. Let $(G_i)_{i\in I}$ be a projective system of topological groups and $(A_i\in \Ab(B_{\underline{G_i}}))_{i\in I}$ a compatible system. Let $\pi_i:G\to G_i$ be the projections. We set $G\coloneqq \underset{\substack{\leftarrow \\ i}}{\lim}\, G_i$ and  $A_{\infty} \coloneqq \underset{\substack{\rightarrow\\i}}{\lim} \, \pi_i^* A_i$. We would like the isomorphism \[\HH^q(B_{\underline{G}},A_{\infty})\cong\lim_{\substack{\rightarrow\\ i}}\, \HH^q(B_{\underline{G_i}},A_i)\] to hold. Unfortunately, this is not always the case even if $G_i$ is finite for all $i$, see \cref{h90failure,vanishing}. In order to recover this property, we introduce pro-condensed groups and their cohomology.

We are adapting \cite[Section 2.6]{Mor2} to the condensed setting.
A pro-object of a category $C$ is a functor $\hat{X}:I\to C$, where $I$ is a cofiltered category. 
\begin{defn}
A pro-condensed group $\hat{G}$ is a pro-object in the category $\Grp(\CC)$ of condensed groups. A pro-condensed group $\hat{G}$ is \emph{strict} if all transition maps $G_i\to G_j$ are epimorphisms of condensed groups.
\end{defn}
Let $\hat{G}:I\to \Grp(\CC)$ be a pro-condensed group. For every $i\in I$ we have a condensed group $G_i$ and a classifying topos $B_{G_i}$, which is a topos over $\CC$. We get $(B_{G_i},f_{ji})_{i,j\in I}$, a projective system of topoi over $\CC$.
\begin{defn}
The classifying topos of a pro-condensed group $\hat{G}:I\to \Grp(\CC)$ is defined as \[
B_{\hat{G}}\coloneqq \underset{\substack{\leftarrow \\ i\in I}}{\lim} \, B_{G_i},
\] where the limit is taken in the 2-category of topoi.
\end{defn}
For all $i\in I$, we call $\pi_i:B_{\hat{G}}\to B_{G_i}$ the projection morphism, and $f_{\hat{G}}:B_{\hat{G}}\to \CC$ the structure morphism.
\begin{defn}
The  $q$th group cohomology of the pro-condensed group $\hat{G}:I\to\Grp(\CC)$ is \[
 \HH^q(B_{\hat{G}},-)\coloneqq R^qf_{\hat{G},*}:\Ab(B_{\hat{G}})\rightarrow \Ab(\CC).
\] \end{defn}
\begin{rmk} Let $\hat{G}$ be a constant pro-condensed group with value $G\in\Grp(\CC)$. Then $B_{\hat{G}}$ is equivalent to $B_G$ via any projection $\pi_i$, and $\HH^q(B_{\hat{G}},-)$ coincides with $\HH^q(B_G,-)$ for all $q$. Hence the condensed cohomology of pro-condensed groups generalises the cohomology of condensed groups.
\end{rmk}
We now prove the existence of a Hochschild-Serre spectral sequence. 

\begin{cns}\label{cns:hs} Let $\hat{H},\hat{G}:J\to\Grp(\CC)$ be two pro-condensed groups, and $\hat{Q}:J\to \Grp(\CC)$ a constant pro-condensed group with value $Q\in\Grp(\CC)$.
Let
\[
1\rightarrow H_j\overset{i_j}{\longrightarrow}G_j\overset{p_j}{\longrightarrow}Q\rightarrow 1
\] be an exact sequence of condensed groups for all $j$, inducing morphisms of pro-condensed groups $\hat{H}\to \hat{G}$ and $\hat{G}\to Q$. By \cite[IV, \S 5.8]{SGA4}, we obtain an equivalence of topoi $B_{H_j}\overset{\sim}{\longrightarrow}B_{G_j}/p_j^*EQ$ for all $j$. Moreover, we have morphisms of topoi $\hat{i}:B_{\hat{H}}\to B_{\hat{G}}$ and $\hat{p}:B_{\hat{G}}\to B_Q$. \end{cns}
\begin{lem}\label{lem:hs}
Let $\hat{H},\hat{G},\hat{Q}$ be as in \cref{cns:hs}. There exists an equivalence of topoi $B_{\hat{H}}\overset{\sim}{\longrightarrow} B_{\hat{G}}/\hat{p}^*EQ$ induced by $B_{H_j}\overset{\sim}{\longrightarrow}B_{G_j}/p_j^*EQ$.
\end{lem}
\begin{proof} By \cref{topoipb}, we have \[
B_{\hat{G}}/\hat{p}^*EQ\cong B_{\hat{G}}\times_{B_Q} B_Q/EQ.
\] In the 2-category of topoi, cofiltered limits commute with fiber products. Thus we have \[ B_{\hat{G}}\times_{B_Q} B_Q/EQ=\underset{\substack{\leftarrow\\j}}{\lim}\, (B_{G_j}\times_{B_Q} B_Q/EQ)\cong \underset{\substack{\leftarrow\\ j}}{\lim} \, B_{G_j}/p_j^*EQ\cong\underset{\substack{\leftarrow\\ j}}{\lim}\, B_{H_j}\cong B_{\hat{H}}. \]
\end{proof}
\begin{prop}[Hochschild-Serre spectral sequence]\label{prop:proHS}
Let $\hat{H},\hat{G},\hat{Q}$ be as in \cref{cns:hs}. Let $A\in\Ab(B_{\hat{G}})$. The condensed abelian group $
\HH^q(B_{\hat{H}}, \hat{i}^* A)$ carries a $Q$-action for all $q$. There is a Hochschild-Serre spectral sequence \[
\HH^p(B_Q,\HH^q(B_{\hat{H}},\hat{i}^*A))\implies \HH^{p+q}(B_{\hat{G}},A).
\]
\end{prop}
\begin{proof}
By \cref{lem:hs} and \cite[Lemma 7]{Flach}, we have the commutative diagram of topoi \[ 
\begin{tikzcd}
B_{\hat{H}}\arrow[rr,"f_{\hat{H}}"]\arrow[d,"\sim"] && \CC\arrow[d, "\sim"]\\
B_{\hat{G}}/\hat{p}^*EQ\arrow[rr,"\hat{p}_{/EQ}"]\arrow[d,"j_{\hat{p}^*EQ}"] && B_Q/EQ\arrow[d,"j_{EQ}"]\\
B_{\hat{G}}\arrow[rr,"\hat{p}"]&& B_Q, \end{tikzcd}
\] where the outer square is a pullback. The composition of the two vertical maps on the left is $\hat{i}$. On the right, the composition gives the morphism $e:\CC\to B_{Q}$. The morphism $e$ is induced by the morphism of groups $\{*\}\to Q$, which sends $*$ to the identity of $Q$. Hence we have $f_{\hat{H},*}\circ \hat{i}^*=e^*\circ \hat{p}_*$. Moreover, since the vertical morphisms are localisation morphisms, we have \[
Rf_{\hat{H},*}\circ \hat{i}^*=e^*\circ R\hat{p}_*.
\]The spectral sequence computing the derived functor of $f_{\hat{G},*}=f_{Q,*}\circ \hat{p}_*$ concludes the proof.
\end{proof}

Finally, we recover the well-known continuity property of cohomology of pro-finite groups, extending it to strict pro-compact Hausdorff groups.
\begin{prop}[Continuity result]\label{prop:cont}
Let $\hat{G}$ be a strict pro-condensed group. Suppose that $G_i$ is compact Hausdorff for all $i$. Let $(A_i,\alpha_{ji})_{i,j\in I}$ be a compatible system of abelian group objects of $(B_{G_i},f_{ij})_{i,j\in I}$ (see \cref{constructionlimittopos}). We set \[
A_{\infty}\coloneqq \underset{\substack{\rightarrow\\ i \in I}}{\lim} \,\pi_i^*A_i\in B_{\hat{G}}.
\] Then the canonical morphism
\[
\underset{\substack{\rightarrow\\ i \in I}}{\lim}\,\HH^q(B_{G_i},A_i)\longrightarrow\HH^q(B_{\hat{G}},A_{\infty})
\] is an isomorphism for any integer $q$.
\end{prop}
\begin{proof}
By \cref{bgstronglycompact}, the topos $B_{G_i}$ is strongly compact for all $i$. Let $f_{ij}:B_{G_i}\to B_{G_j}$ be a transition map. Let $K$ be the kernel of $G_i\twoheadrightarrow G_j$, which is compact Hausdorff. The localisation of $f_{ij}$ at $EG_j$ is $B_K\rightarrow \CC$. This morphism is strongly compact by \cref{bgstronglycompact}. Hence, by \cref{sclocalisation}, so is $f_{ij}$. We conclude by \cref{lem:3}.
\end{proof}

\subsection{The category \texorpdfstring{$\DDD^+(B_{\hat{G}})$}{DBhatG}}\label{dbhatg}
Let $\hat{G}:I\rightarrow \Grp(\CC)$, $i\mapsto G_i$ be a strict pro-group. Suppose that all $G_i$ are compact Hausdorff. In this section, we give a description of objects in $\DDD^+(B_{\hat{G}})$ in terms of objects of $\DDD^+(B_{G_i})$ for all $i$ (see \cref{lem:colimitdescription}). Moreover, we obtain a formula which relates the internal $\Hom$ in $\DDD(B_{\hat{G}})$ with internal $\Hom$'s in $\DDD(B_{G_i})$ for all $i$ (see \cref{internalhomdescr}). To do this, we start by studying the defining site of the category $B_{\hat{G}}$. 

Let $G=\lim_i G_i$ be the topological group associated to $\hat{G}$. We denote by $G-\Top^c$ the category of compact Hausdorff topological spaces with a continuous action of $G$, with $G$-equivariant continuous maps as morphisms.
\begin{prop}\label{definingsite}
Let $\hat{G}-\Top^c$ be the full subcategory of $G-\Top^c$ of those compact Hausdorff topological spaces with a continuous action of $G$ which factors through $G_i$ for some $i$. Let $j_{cond}$ be the topology on $\hat{G}-\Top^c$ with finitely jointly surjective families of maps as covers. Then there is a natural morphism \[
B_{\hat{G}}\rightarrow \Sh(\hat{G}-\Top^c,j_c)
\] which is an equivalence of topoi.
\end{prop}
\begin{proof}
For all $i$, let $G_i-\Top^c$ be the category of topological spaces with a continuous action of $G_i$, with $G_i$-equivariant continuous morphisms. Let $j_i$ be the coarsest topology such that the forgetful functor $(G_i-\Top^c,j_i)\to (\Top^c,j_{cond})$ is continuous. We have a defining site $(G_i-\Top^c,j_i)$ for $B_{G_i}$. By \cite[VI, \S 8.2.3]{SGA4}, a site for $B_{\hat{G}}$ is given by $(\lim_{\rightarrow\, I}\, (G_i-\Top^c),j)$. Here $j$ is the coarsest topology such that the functor $(G_i-\Top^c,j_i)\rightarrow (\lim_{\rightarrow\, I} \,(G_i-\Top^c),j)$ is continuous for all $i$.

We can make the category $\lim_{\rightarrow\, I}\, (G_i-\Top^c)$ explicit as follows. An object of this category is a compact Hausdorff topological group with a continuous action of $G_i$ for some $i$. Let $X_1,X_2$ be two objects of this category with an action of $G_1$, $G_2$ respectively. Then there exists $k\in I$ such that $G_k$ surjects on both $G_1$ and $G_2$. Then a morphism $X_1\to X_2$ is a $G_k$-equivariant continuous map $X_1\to X_2$. Hence we have an equivalence of categories \[
\lim_{\rightarrow\, I}\, (G_i-\Top^c)\cong \hat{G}-\Top^c.
\] Under this identification,  $j_{cond}$ is the coarsest topology on $\hat{G}-\Top^c$ such that the functor \[
(G_i-\Top^c,j_i)\rightarrow (\hat{G}-\Top^c,j_{cond})
\] is continuous for all $i$. The result follows.
\end{proof} 
\begin{lem}\label{lem:sectioncheck}
Let $\alpha:\FF\to \GG$ be a morphism in $\DDD^+(B_{\hat{G}})$. Suppose that for all $i\in I$ the induced morphism $R\GGamma(B_{\hat{G}}/EG_i,\FF_{|EG_i})\rightarrow R\GGamma(B_{\hat{G}}/EG_i,\GG_{|EG_i})$ is an equivalence in $\DDD(\CC)$. Then $\alpha$ is an equivalence.
\end{lem}
\begin{proof} 
Let $X$ be a compact Hausdorff topological space with an action of $G_i$, for some $i$. We show that the morphism $
R\Gamma(B_{\hat{G}}/X,\FF_{|X})\rightarrow R\Gamma(B_{\hat{G}}/X,\GG_{|X})$ is an equivalence in $\DDD^+(\Z)$. We observe that \[
EG_i\times X^{tr}\rightarrow X, \quad (g,x)\mapsto g\cdot x,
\] is a covering in $\hat{G}-\Top^c$, where $X^{tr}$ is $X$ with the trivial action of $G$. Thus we have \[
R\Gamma(B_{\hat{G}}/X,\FF_{|X})=\mathrm{colim}(R\Gamma(B_{\hat{G}}/(EG_i\times X^{tr})^{\bullet},\FF_{|(EG_i\times X^{tr})^{\bullet}}),
\] and similarly for $\GG$. For all $n$, we have an isomorphism $(EG_i\times X^{tr})^n\cong EG_i \times (G_i^{tr})^{n-1}\times X^{tr}$ over $X$. Thus it is enough to check that the morphism \[
R\Gamma(B_{\hat{G}}/EG_i\times (G_i^{tr})^n\times X^{tr},\FF_{|EG_i\times (G_i^{tr})^n\times X^{tr}})\rightarrow R\Gamma(B_{\hat{G}}/EG_i\times (G_i^{tr})^n\times X^{tr},\GG_{|EG_i\times (G_i^{tr})^n\times X^{tr}})
\] is an equivalence for all $n$ and all $i$. This follows from the hypothesis. Indeed we have \[
R\Gamma(B_{\hat{G}}/EG_i\times (G_i^{tr})^n\times X^{tr},\FF_{|EG_i\times (G_i^{tr})^n\times X^{tr}})=R\Gamma(G_i^n\times X,R\GGamma(B_{\hat{G}}/EG_i,\FF_{|EG_i})),
\] and similarly for $\GG$. 
Thus for all $X$ compact Hausdorff with an action of $G_i$ for some $i$, the morphism $R\Gamma(B_{\hat{G}}/X,\FF_{|X})\rightarrow R\Gamma(B_{\hat{G}}/X,\GG_{|X})$ is an equivalence in $\DDD^{+}(\Z)$. The result follows from \cite[Corollary 2.1.2.3.]{SAG} and \cref{definingsite}. 
\end{proof}
\begin{ntt} For all $i\in I$, let $\pi_i:\hat{G}\to G_i$ be the projection. For all $j\to i$, let $\rho_{i,j}$ be the morphism $G_j\twoheadrightarrow G_i$. We set $\hat{U}_i\coloneqq \mathrm{ker}(\pi_i)$, and $K_{i,j}\coloneqq \mathrm{ker}(\rho_{i,j})$. Thus we have \[B_{\hat{G}}/EG_i\cong B_{\hat{U}_i}=\lim_{\substack{\leftarrow\\ j\to i}}\, B_{K_{i,j}}.\]  We call $\pi_j^i$ the projections $B_{\hat{U}_i}\to B_{K_{i,j}}$.
\end{ntt}

\begin{prop}\label{lem:colimitdescription}
For all $\FF\in \DDD^+(B_{\hat{G}})$, the natural morphism $
\underset{\substack{\rightarrow \\ j}}{\lim}\, \pi_j^*R\pi_{j,*} \FF \rightarrow \FF$ is an equivalence.
\end{prop}
\begin{proof}
By \cref{lem:sectioncheck}, it is enough to prove that the morphism \begin{equation}\label{eqn:lem:morph}
R\GGamma(B_{\hat{G}}/EG_i,(\lim_{\rightarrow\, j} \pi_j^*R\pi_{j,*}\FF)_{|EG_i})\rightarrow R\GGamma(B_{\hat{G}}/EG_i,\FF_{|EG_i})
\end{equation} is an equivalence for all $i$. This can be written firstly as \[
R\GGamma(B_{\hat{U}_i},\lim_{\substack{\rightarrow\\ j\twoheadrightarrow i}} \pi_j^{i,*}(R\pi_{j,*}\FF)_{|EG_i})\rightarrow R\GGamma(B_{\hat{U}_i},\FF_{|EG_i}),
\] 
and then, by \cref{prop:cont}, as \[
\lim_{\substack{\rightarrow\\ j\twoheadrightarrow i}} R\GGamma(B_{K_{i,j}},(R\pi_{j,*}\FF)_{|EG_i})\rightarrow R\GGamma(B_{\hat{U}_i},\FF_{|EG_i}).
\]
For all $j$, we have $R\GGamma(B_{K_{i,j}},(R\pi_{j,*}\FF)_{|EG_i})=(R\rho_{i,j,*}R\pi_{j,*}\FF)_{|EG_i}=(R\pi_{i,*}\FF)_{|EG_i}$. Hence the colimit is constant and \eqref{eqn:lem:morph} becomes \[
(R\pi_{i,*}\FF)_{|EG_i}\rightarrow R\GGamma(B_{\hat{U}_i},\FF_{|EG_i}).
\] This is an equivalence by \cref{topoipb}.
\end{proof}
\begin{cor}
Let $\FF\in\DDD^+(B_{\hat{G}})$ such that $R\pi_{i,*}\FF=0$ for all $i$. Then $\FF=0$.
\end{cor}

When working with $B_{\hat{G}}$ instead of $B_G$ we lose an important property. Indeed, the morphism of topoi $\CC\to B_{\hat{G}}$ is not a localisation morphism. Thus, we can't check vanishing in $\CC$ apriori. However, thanks to \cref{lem:colimitdescription} we can reduce to prove vanishing in $B_{G_i}$, where we have localisation morphisms $\CC\to B_{G_i}$.
\begin{rmk}\label{rmk:morphismofinternalhom}
Let $f:\mathcal{T}_1\to \mathcal{T}_2$ be a morphism of topoi. For $i=1,2$, we denote by $R\underline{\Hom}_{\mathcal{T}_i}(-,-)$ the derived internal $\Hom$ in $\DDD(\mathcal{T}_i)$. For all $\FF,\GG\in \DDD(\mathcal{T}_2)$ we have a canonical morphism \begin{equation}\label{eqn:morphismofinternalhom}
f^*R\underline{\Hom}_{\mathcal{T}_2}(A,B)\rightarrow R\underline{\Hom}_{\mathcal{T}_1}(f^*A,f^*B),
\end{equation} which is functorial in $A$ and $B$. Moreover, if $f$ is a localisation morphism, then \eqref{eqn:morphismofinternalhom} is an isomorphism.
\end{rmk}

\begin{lem}\label{lem:internalhoms}
Let $\FF\in \DDD^+(B_{G_i})$ and let $\GG\in \DDD^+(B_{\hat{G}})$. Then the natural morphism \[
R\underline{\Hom}_{B_{G_i}}(\FF,R\pi_{i,*}\GG)\rightarrow R\pi_{i,*}R\underline{\Hom}_{B_{\hat{G}}}(\pi_i^*\FF,\GG)
\] is an isomorphism.
\end{lem}
\begin{proof}
Let us consider the pull-back diagram of topoi given by the localisation at $EG_i$ \[
\begin{tikzcd}
B_{\hat{U}_i}\ar[r,"u"]\ar[d,"\iota_i"]&\CC\ar[d,"e"]\\
B_{\hat{G}}\ar[r,"\pi_i"]& B_{G_i}.
\end{tikzcd}
\] 
By \cref{rmk:localizediso,rmk:morphismofinternalhom}, it is enough to show that the morphism\[
R\underline{\Hom}_{\CC}(e^*\FF,Ru_*\GG)\rightarrow Ru_*R\underline{\Hom}_{B_{\hat{U}_i}}(u^*e^*\FF,\GG)
\] is an equivalence in $\DDD(\CC)$. This can be checked on extremally disconnected topological spaces. Take $S$ extremally disconnected. Then, by \cref{rmk:morphismofinternalhom}, the morphism
 \[
 R\Gamma(\CC/S,R\underline{\Hom}_{\CC}(e^*\FF,Ru_*\GG)_{|S})\rightarrow R\Gamma(\CC/S,Ru_*R\underline{\Hom}_{B_{\hat{U}_i}}(u^*e^*\FF,\GG)_{|S})
 \] is \[
R\Hom_{\CC/S}(e_{/S}^*\FF_{|S},Ru_{/S,*}(\GG_{|S}))\rightarrow R\Hom_{B_{\hat{U}_i}/S}(u_{/S}^*e_{/S}^*\FF_{|S},\GG_{|S}).
\] This is an equivalence by adjunction. The result follows.
\end{proof}
\begin{cor}\label{internalhomdescr} Let $\FF,\GG\in \DDD^+(B_{\hat{G}})$. We have the following expression of $R\underline{\Hom}_{\hat{G}}(\FF,\GG)$. \begin{enumerate}[(i)]
\item If $\FF=\pi_i^* \FF'$, with $\FF'\in \DDD^+(B_{G_i})$, then we have \[
R\underline{\Hom}_{B_{\hat{G}}}(\FF,\GG)\cong \underset{\substack{\rightarrow\\ j\to i}}{\lim}\, \pi_j^*R\underline{\Hom}_{B_{G_j}}(\rho_{i,j}^*\FF',R\pi_{j,*}\GG).\] 
\item In general, we have \[
R\underline{\Hom}_{B_{\hat{G}}}(\FF,\GG)\cong R\underset{\substack{\leftarrow\\ i}}{\lim}\, \underset{\substack{\rightarrow\\ j\to i}}{\lim}\,\pi_j^*R\underline{\Hom}_{B_{G_j}}( \rho_{i,j}^*R\pi_{i,*}\FF,R\pi_{j,*}\GG).
\]
\end{enumerate}
\end{cor}
\begin{proof}
For (i), we apply \cref{lem:colimitdescription} to $R\Hom_{B_{\hat{G}}}(\FF,\GG)$. We conclude by applying \cref{lem:internalhoms} to $R\pi_{j,*}R\Hom_{B_{\hat{G}}}(\FF,\GG)$.

For (ii), we apply \cref{lem:colimitdescription} firstly on $\FF$ and then on $R\underline{\Hom}_{B_{\hat{G}}}(\pi_i^*R\pi_{i,*}\FF,\GG)$ for all $i$. We conclude by applying \cref{lem:internalhoms} to $R\pi_{j,*}R\underline{\Hom}_{B_{\hat{G}}}(\pi_i^*R\pi_{i,*}\FF,\GG)$  for all $j$. \end{proof}

\section{Cohomology of \texorpdfstring{$\hat{W}_F$}{WF}}\label{section2}
\begin{defn}
Let $k$ be a finite field. The Weil group $W_k\subset G_k$ is defined by the pullback square of topological groups \[
\begin{tikzcd}
W_k\arrow[r,hook] \arrow[d,"\sim"] & G_k \arrow[d,"\sim"]\\
\Z \arrow[r,hook] & \hat{\Z}.
\end{tikzcd}
\]
\end{defn} The cohomology of the condensed group $W_k\cong\Z$ is explicited in \cref{cohomologyz}.

Let $F/\Q_p$ be a finite field extension, and $k\coloneqq \mathcal{O}_F/\mathfrak{m}_F$ the finite residue field. Let $F^{ur}$ be its maximal unramified extension, and let $L$ be the completion of $F^{ur}$. Let $\overline{F}$ be an algebraic closure of $F$ and let $\overline{L}$ be an algebraic closure of $L$ containing $\overline{F}$. Let $G_F\coloneqq \Gal(\overline{F}/F)$ and $G_k\coloneqq \Gal(F^{ur}/F)=\Gal(\overline{k}/k)$. Let $I\coloneqq \Gal(\overline{F}/F^{ur})$ be the inertia group.
We have an exact sequence of topological groups $1\rightarrow I\rightarrow G_F\rightarrow G_k\rightarrow 1$.
\begin{defn}
The \emph{Weil group} of the local field $F$ is the pullback of $W_k$ under $G_F\twoheadrightarrow G_k$ \[
W_F\coloneqq G_F\times_{G_k}W_k.
\] The pullback is taken in the category of topological groups.
\end{defn}
The subgroup $I\subset W_F$ is clopen, and we have an exact sequence of topological groups \[
1\rightarrow I\rightarrow W_F\rightarrow W_k\rightarrow 1.
\] The topological groups $I$ and $W_F$ are a profinite group and a prodiscrete group respectively. Indeed, if $\SS$ denotes the set of open normal subgroups of $I$, we have \begin{equation}\label{profinprodisc}
I=\underset{\substack{\leftarrow\\ U\in \SS}}{\lim}\, I/U, \qquad W_F=\underset{\substack{\leftarrow\\ U\in \SS}}{\lim}\, W_F/U,
\end{equation} where $I/U$ is finite and $W_F/U$ is discrete for all $U$. 

It follows from Krasner's Lemma applied to the extension $F^{ur}\subset L$ that $I=\Gal(\overline{L}/L)$. Hence we have a continuous action of the topological group $I$ on the discrete abelian group $\overline{L}^{\times}$. By Hilbert 90 we have \begin{equation}\label{discretehilbert90}
\H^1(B_I(\Set),\overline{L}^{\times})=0.
\end{equation} We would like a similar computation for the \emph{topological} abelian group $\overline{L}^{\times}$, where the topology is induced by the natural topology on $\overline{L}$. We try to recover this in $B_I$, the classifying topos of the condensed group $I$.
\begin{defn}\label{overlinel}
Let $K/L$ be a finite Galois extension of group $G$. Let us endow $K^{\times}$ with the topology induced by the inclusion $K^{\times}\subset K$. Let $\underline{K^{\times}}$ be the associated condensed $G$-module. We set \[
\overline{L}^{\times}\coloneqq \lim_{\substack{\rightarrow\\ K}} \underline{K^{\times}} \in \Ab(B_I),
\] where $\underline{K^{\times}}$ is seen as an object of $\Ab(B_I)$ by pullback along $B_{I}\to B_G$, and the colimit is computed in $B_I$.
\end{defn}
\begin{prop}\label{h90failure}
The abelian group $\HH^1(B_I,\overline{L}^{\times})(*)$ is not torsion. 
\end{prop}
\begin{proof}
Let us consider the exact sequence in $\Ab(B_I)$ \[
0\rightarrow \underline{L^{\times}} \rightarrow \overline{L}^{\times}\rightarrow \overline{L}^{\times}/\underline{L^{\times}}\rightarrow 0.
\] We get an exact sequence in cohomology \[
\dots\rightarrow\HH^0(B_I,\overline{L}^{\times}/\underline{L^{\times}})(*)\overset{\alpha}{\longrightarrow} \HH^1(B_I,\underline{L^{\times}})(*)\overset{\beta}{\longrightarrow} \HH^1(B_I,\overline{L}^{\times})(*)\rightarrow \dots.
\] 
Since $\underline{L^{\times}}$, $\overline{L}^{\times}$ and $\overline{L}^{\times}/\underline{L^{\times}}$ are solid, the spectral sequence given by \cref{cartanleray} degenerates. Hence \[
\HH^q(B_I,\underline{L^{\times}})(*)=H^q(\underline{L^{\times}}(*)\rightarrow \underline{L^{\times}}(I)\rightarrow\dots ),
\] and similarly for $\overline{L}^{\times}$ and $\overline{L}^{\times}/\underline{L^{\times}}$. We have the following morphisms of complexes \[
\begin{tikzcd}
0\ar[d]&0\ar[d]& \\
\underline{L^{\times}}(*)\ar[r]\ar[d]&\Cont(I,L^{\times})\ar[r]\ar[d]&\dots\\
\overline{L}^{\times}(*)\ar[r]\ar[d]&\Cont(I,\overline{L}^{\times})\ar[r]\ar[d]&\dots\\
\overline{L}^{\times}/\underline{L^{\times}}(*)\ar[r]\ar[d]&\overline{L}^{\times}/\underline{L^{\times}}(I)\ar[d]\ar[r]&\dots.\\
0& 0 & 
\end{tikzcd}
\] Since for all $n$ the topological group $I^n$ is profinite, we have \[\Ext^1_{\Ab(\CC)}(\Z[I^n],\underline{L^{\times}})=\Ext^1_{\Solid}(\Z[I^n]^{\blacksquare},\underline{L^{\times}})=0.\] Hence vertical sequences are exact. By diagram chasing, the morphism $\alpha$ is explicited as follows
\[
\alpha: \HH^0(B_I,\overline{L}^{\times}/\underline{L^{\times}})(*)\rightarrow \HH^1(B_I,\underline{L^{\times}})(*), \qquad \overline{x}\mapsto (i\mapsto x^{-1}i\cdot x),
\] where $x$ is a representative of $\overline{x}$ in $\overline{L}^{\times}(*)=\bigcup_{K}\underline{K^{\times}}(*)$. In particular, since $x\in K^{\times}$ for some $K$, its orbit under the action of $I$ is finite. Hence, the image of $\alpha$ lands in those continuous homomorphisms $I\to L^{\times}$ whose image is finite.
Local Class Field theory provides an isomorphism of topological groups (\cite[Corollary 9.16]{Har}) \[
J\cong \O_F^{\times},
\] where $J=I/\Gal(\overline{F}/F^{ab})$. Then we get a continuous group homomorphism $\varphi:I\twoheadrightarrow J \cong \O_F^{\times}\hookrightarrow L^{\times}$. For all $n$, the image of $\varphi^n$ is $(\O_F^{\times})^n$, an infinite subgroup. Hence $\varphi^n\notin \mathrm{im}(\alpha)=\mathrm{ker}(\beta)$ for all $n$. Thus $\beta(\varphi)$ is an element of $\HH^1(B_I,\overline{L}^{\times})(*)$ which is not torsion.
\end{proof}
Consequently, the condensed version of Hilbert 90 does not hold if we consider $I$ as a condensed group. Hence we consider the profinite topological group $I$ (resp.\ the prodiscrete topological group $W_F$) as a pro-condensed group, say $\hat{I}$ (resp.\ $\hat{W}_F$). We get an exact sequence of pro-condensed groups \[
 1\rightarrow \hat{I} \rightarrow \hat{W}_F \rightarrow \Z \rightarrow 1.
 \]
Following \cref{progroups}, we have classifying topoi $B_{\hat{I}}$ and $B_{\hat{W}_F}$. Let $M\in\Ab(B_{\hat{W}_F})$. By \cref{prop:proHS,cohomologyz}, we get an exact sequence \begin{equation}\label{eqn:degeneratehs}
 0\rightarrow\HH^1(B_{W_k},\HH^{q-1}(B_{\hat{I}},M))\rightarrow\HH^q(B_{\hat{W}_F},M)\rightarrow\HH^0(B_{W_k},\HH^q(B_{\hat{I}},M))\rightarrow 0
 \end{equation} for all $q$.
 
In the next section we see how to recover a topological version of \eqref{discretehilbert90} in this setting. 

\subsection{The complex \texorpdfstring{$R\GGamma(B_{\hat{I}},\overline{L}^{\times})$}{RGammaBIL}}
Replacing $I$ with $\hat{I}$ in \cref{overlinel}, we obtain an object \[\overline{L}^{\times}\coloneqq \underset{\substack{\rightarrow\\ K}}{\lim}\, \underline{K^{\times}} \in \Ab(B_{\hat{I}}).\] The goal of this section is to compute the complex $R\GGamma(B_{\hat{I}},\overline{L}^{\times})$. In particular, we show that it is concentrated in cohomological degree $0$.
 
If $K/L$ is a finite Galois extension, the ring $\O_K$ is a discrete valuation ring with residue field $\overline{k}$. The group of invertible elements of $\O_K$ is $\O_K^{\times}\coloneqq \{x\in \O_K \, | \, v_K(x)= 0\}$. For all $i\ge 1$, we set $\UU_K^i\coloneqq 1+\mathfrak{m}_K^i$.
Let us consider the filtration of the topological abelian group $K^{\times}$
\begin{equation}\label{kfiltration}
\dots\subset \UU_K^{i+1}\subset\UU_K^i\subset \dots \subset \UU_K^0\coloneqq \O_K^{\times}\subset \UU_K^{-1}\coloneqq K^{\times} .
\end{equation} This is a filtration by clopen subgroups of $K^{\times}$. Since $K^{\times}$ is complete with the topology induced by the valuation, we have \[
K^{\times}\cong \underset{\substack{\leftarrow \\ i\ge 0}}{\lim} \, K^{\times}/\UU_K^i
\] as topological abelian groups. Moreover, by \cite[IV.2, Proposition 6]{serreCL} we have the following associated graded topological abelian groups \[
gr^{-1}=\UU_K^{-1}/\UU_K^0\cong \Z, \qquad gr^0=\UU_K^0/\UU_K^{-1}\cong \overline{k}^{\times}, \qquad gr^i=\UU_K^i/\UU_K^{i-1}\cong \overline{k} \quad \forall i\ge 1,
\] where the last isomorphism is non-canonical. \begin{lem}\label{lem:derivedprodiscreteness}
The filtration \eqref{kfiltration} is a filtration by clopen subgroups, and we have \[
\underline{K^{\times}}\cong R\lim_{\substack{\leftarrow \\ i\ge 0}} \underline{K^{\times}/\UU_K^i} \in \DDD(B_G),
\] where the right hand side is a pro-discrete condensed abelian group with its obvious $G$-action.
\end{lem}
\begin{proof}
We already observed that the subgroup $\UU_K^i\subset K^{\times}$ is clopen for all $i$. Since the functor $\underline{(-)}:\Top \to \CC$ commutes with limits, the morphism \begin{equation}\label{ausiliario}
\underline{K^{\times}}\rightarrow \lim_{\substack{\leftarrow\\i\ge 0}} \underline{K^{\times}/\UU_K^i}
\end{equation} is an isomorphism in $\Ab(\CC)$. Moreover, for all $i$ the transition morphism $K^{\times}/\UU_K^{i+1}\twoheadrightarrow K^{\times}/\UU_K^i$ is $G$-equivariant. Hence \eqref{ausiliario} is an isomorphism in $\Ab(B_G)$ as well. For all $i$, the topological $G$-module $K^{\times}/\UU_K^i$ is discrete by the first part of the lemma. Moreover, all the transition morphisms are $G$-equivariant continuous surjections.
Thus, for all $i$ the morphism $\underline{K^{\times}/\UU_K^{i+1}}\twoheadrightarrow \underline{K^{\times}/\UU_K^i}$ is an epimorphism in $\CC$, and consequently in $B_G$. The result follows from \cite[Proposition 3.1.10]{proetale} and \cref{prop:bgreplete}. \end{proof}

\begin{lem}\label{lem:torsioncohomologygroups}
For all $S$ extremally disconnected and all $q\ge1$, the abelian group $\HH^q(B_G,\underline{K^{\times}})(S)$ is torsion. Hence $\HH^q(B_{\hat{I}},\overline{L}^{\times})(S)$ is torsion.
\end{lem}
\begin{proof}
By \cref{prop:cont} we have \[R\GGamma(B_{\hat{I}},\overline{L}^{\times})\cong\lim_{\substack{\rightarrow\\ K}} R\GGamma(B_G,\overline{K^{\times}}).\] Hence it is enough to check that \[
(\tau^{\ge 1}R\GGamma(B_G,\underline{K^{\times}}))(S)\cong \tau^{\ge 1}(R\GGamma(B_G,\underline{K^{\times}})(S)) \]
 has torsion cohomology groups for any $K$ and any $S$ extremally disconnected. By \cref{prop:discretegroupsections} we have
\[
R\GGamma(B_G,\underline{K^{\times}})(S)\cong R\Gamma(B_G(\Set),\underline{K^{\times}}(S)).
\] The result follows since higher cohomology groups of the finite group $G$ are killed by the order of $G$.
\end{proof}
\begin{lem}\label{lem:equivalencemodm}
For any finite Galois $K/L$ of group $G$, the canonical map \[
K^{\times,\delta}\otimes^L \Z/m\Z \overset{\sim}{\longrightarrow} \underline{K^{\times}}\otimes^L \Z/m\Z
\] is an equivalence in $\DDD(B_G)$, where we consider $K^{\times,\delta}$ as a discrete abelian group with a continuous $G$-action.
\end{lem}
\begin{proof}
Let $e$ be ramification index of $K$, i.e.\ the valuation of $p\in K^{\times}$. We set $e_1\coloneqq e/(p-1)$ and we consider the finite filtration \[
\dots=\UU_K^{n+1}=\UU_K^n\subset \UU_K^{n-1}\subset\dots\subset \UU_K^0\coloneqq \underline{\O_K^{\times}}\subset \UU_K^{-1}=\underline{K^{\times}}
\] of $\underline{K^{\times}}$ for some $n>e_1$. Similarly, we have a finite filtration \[
\dots=\UU_K^{n+1,\delta}=\UU_K^{n,\delta}\subset \UU_K^{n-1,\delta}\subset\dots\subset \UU_K^{0,\delta}\coloneqq \O_K^{\times,\delta}\subset \UU_K^{-1,\delta}=K^{\times,\delta}
\] of $K^{\times,\delta}$. For any $-1\le i\le n$ the map $\UU_K^{i,\delta}/\UU_K^{i+1,\delta}\rightarrow \UU_K^i/\UU_K^{i+1}$is an isomorphism of discrete abelian groups, hence the map \[
(\UU_K^{i,\delta}/\UU_K^{i+1,\delta})\otimes^L \Z/m\Z \rightarrow (\UU_K^i/\UU_K^{i+1})\otimes^L \Z/m\Z
\] is an equivalence. Therefore, it is enough to check that \begin{equation}\label{eqn:lem:m}
\UU_K^{n,\delta}\otimes^L \Z/m\Z \rightarrow \UU_K^n\otimes^L \Z/m\Z
\end{equation} is an equivalence. If $m$ is coprime to $p$, \[
\UU_K^n\overset{(-)^m}{\longrightarrow} \UU_K^n
\] is an isomorphism of topological groups (see \cref{lem:mcoprime}) so that both sides of \eqref{eqn:lem:m} vanish. Hence we may suppose $m=p^{\nu}$. Then the map \[
\UU_K^n\overset{(-)^{p^{\nu}}}{\longrightarrow} \UU_K^n
\] induces an isomorphism of topological abelian groups \[
\UU_K^n\overset{(-)^{p^{\nu}}}{\longrightarrow} \UU_K^{n+\nu e}
\] onto the open subgroup $\UU_K^{n+\nu e}\subset \UU_K^n$, see \cite[Corollaire 1]{SerreAlgClos}. We obtain an isomorphism \[
\UU_K^n\otimes^L \Z/p^{\nu}\Z \cong \UU_K^n/\UU_K^{n+\nu e},
\]where $\UU_K^n/\UU_K^{n+\nu e}$ is discrete. Hence the map \[
\UU_K^{n,\delta}\otimes^L \Z/p^{\nu}\Z \cong \UU_K^{n,\delta}/\UU_K^{n+\nu e,\delta} \overset{\sim}{\longrightarrow} \UU_K^n/\UU_K^{n+\nu e}\cong \UU_K^n\otimes^L \Z/p^{\nu}\Z
\] is an equivalence. The result follows.
\end{proof}
\begin{cor}
For any positive integer $m$, one has an exact sequence \begin{equation}\label{exactseq}
0\rightarrow \mu_m(\overline{L})\longrightarrow \overline{L}^{\times}\overset{(-)^{m}}{\longrightarrow}\overline{L}^{\times}\rightarrow 0
\end{equation} in $\Ab(B_{\hat{I}})$.
\end{cor}
\begin{lem}\label{lem:mcoprime}
Let $n\in \N$. For $m$ coprime to $p$, the continuous homomorphism \[
(-)^m:\UU_K^n\rightarrow \UU_K^n
\] is an isomorphism of topological groups.
\end{lem}
\begin{proof}
The topological abelian group $\UU_K^{n}$ is complete for the topology induced by the valuation, hence we have \[
    \UU_K^{n}=\lim_{\substack{\leftarrow \\ i\ge n}} \UU_K^{n}/\UU_K^{i}.
    \]
For all $i\ge 1$, $\UU_K^i/\UU_K^{i+1}\cong \overline{k}$ and $(-)^m$ is the multiplication $\cdot m: \overline{k}\to\overline{k}$. This is an isomorphism since $m$ is coprime with $p$. The result follows from \cite[V.1, Lemma 2]{serreCL}. \end{proof}

\begin{prop}\label{vanishing}
For any finite Galois extension $K/L$ of group $G$ we have $R\GGamma(B_G,\underline{K^{\times}})\cong \underline{L^{\times}}[0]$ and consequently $R\GGamma(B_{\hat{I}},\overline{L}^{\times})\cong \underline{L^{\times}}[0]$.
\end{prop}
\begin{proof}
By \cref{lowercohomologygroups}, 1), $\HH^0(B_G,K^{\times,\delta})=L^{\times,\delta}$ and $\HH^0(B_G,\underline{K^{\times}})=\underline{L^{\times}}$. Hence we have a morphism of fiber sequences \[
\begin{tikzcd}
L^{\times,\delta}\ar[r]\ar[d]& R\GGamma(B_G,K^{\times,\delta})\ar[d]\ar[r] & \tau^{\ge 1} R\GGamma(B_G,K^{\times,\delta})\ar[d]\\
\underline{L^{\times}}\ar[r] &R\GGamma(B_G,\underline{K^{\times}})\ar[r]&\tau^{\ge 1} R\GGamma(B_G,\underline{K^{\times}}).
\end{tikzcd}
\] Applying $(-)\otimes^L \Z/m\Z$, we obtain \[
\begin{tikzcd}
L^{\times,\delta}\otimes^L \Z/m\Z \ar[r]\ar[d]& R\GGamma(B_G,K^{\times,\delta})\otimes^L \Z/m\Z\ar[d]\ar[r] & \tau^{\ge 1} R\GGamma(B_G,K^{\times,\delta})\otimes^L \Z/m\Z\ar[d]\\
\underline{L^{\times}}\otimes^L \Z/m\Z\ar[r] &R\GGamma(B_G,\underline{K^{\times}})\otimes^L \Z/m\Z\ar[r]&\tau^{\ge 1} R\GGamma(B_G,\underline{K^{\times}})\otimes^L \Z/m\Z,
\end{tikzcd}
\] where the left and the middle vertical maps are equivalences by \cref{lem:equivalencemodm}. Hence the right vertical map is an equivalence as well. By \cref{lowercohomologygroups}, 3) and \cite[Proposition 5, Proposition 8, \S 3.3, c)]{serreGC} we have $\tau^{\ge 1}R\GGamma(B_G,K^{\times,\delta})=\tau^{\ge 1}\underline{R\Gamma(B_G(\Set),K^{\times,\delta})}=0$. Consequently, one gets \[
(\tau^{\ge 1}R\GGamma(B_G,\underline{K^{\times}}))\otimes^L \Z/m\Z =0.
\] Therefore, for all extremally disconnected $S$, we have \[
(\tau^{\ge 1}(R\GGamma(B_G,\underline{K^{\times}})(S)))\otimes^L \Z/m\Z \cong ((\tau^{\ge 1}R\GGamma(B_G,\underline{K^{\times}}))\otimes^L \Z/m\Z)(S)=0.
\] Hence we get\[
(\tau^{\ge 1}(R\GGamma(B_G,\underline{K^{\times}})(S)))\otimes^L \Q/\Z \cong \lim_{\rightarrow \, m} (\tau^{\ge 1}(R\GGamma(B_G,\underline{K^{\times}})(S)))\otimes^L \Z/m\Z=0.
\] Moreover, we have \[
(\tau^{\ge 1}(R\GGamma(B_G,\underline{K^{\times}})(S)))\otimes^L \Q=0
\] by \cref{lem:torsioncohomologygroups}. Thus the fiber sequence \[
\tau^{\ge 1}(R\GGamma(B_G,\underline{K^{\times}})(S)) \rightarrow (\tau^{\ge 1}(R\GGamma(B_G,\underline{K^{\times}})(S)))\otimes^L \Q \rightarrow (\tau^{\ge 1}(R\GGamma(B_G,\underline{K^{\times}})(S)))\otimes^L \Q/\Z
\] yields\[
(\tau^{\ge 1}R\GGamma(B_G,\underline{K^{\times}}))(S) \cong \tau^{\ge 1}(R\GGamma(B_G,\underline{K^{\times}})(S)) =0
\] for all extremally disconnected $S$. We get $\tau^{\ge 1} R\GGamma(B_G,\underline{K^{\times}})=0$, hence we have \[
R\GGamma(B_G,\underline{K^{\times}})\cong \HH^0(B_G,\underline{K^{\times}})[0]\cong \underline{L^{\times}}[0].
\]
\end{proof}

\subsection{The complex \texorpdfstring{$R\GGamma(B_{\hat{W}_F},\R/\Z(1))$}{RGammaBWRZ1}}
In this section we define a ``dualising object" in the category $\DDD^b(B_{\hat{W}_F})$, which we denote by $\R/\Z(1)$. We study the complex $R\GGamma(B_{\hat{W}_F},\R/\Z(1))$, showing that it is concentrated in cohomological degrees $0,1,2$. 
\begin{lem}\label{OLcohom}
We have $R\GGamma(B_{W_k},\underline{\O_L^{\times}})=\underline{\O_F^{\times}}[0]$.
\end{lem}
\begin{proof}
Let us consider the filtrations \[
\dots\subset \UU_L^i\subset \dots\subset \UU_L^2\subset \UU_L^1\subset \UU_L^0 \coloneqq \O_L^{\times}
\] and \[
\dots\subset \UU_F^i\subset \dots\subset \UU_F^2\subset \UU_F^1\subset \UU_F^0 \coloneqq \O_F^{\times}.
\] 
As in \cref{lem:derivedprodiscreteness}, we have \[
 \underline{\O_F^{\times}}\cong R\lim_{\substack{\leftarrow\\ i\ge 1}} \underline{\O_F^{\times}/\UU_F^i} \in \DDD(\CC), \qquad 
\underline{\O_L^{\times}}\cong R\lim_{\substack{\leftarrow\\ i\ge 1}} \underline{\O_L^{\times}/\UU_L^i} \in \DDD(B_{W_k}).
\] 
By \cref{cohomologyz}, we have \[
R\GGamma(B_{W_k},gr^0(\underline{\O_L^{\times}}))=k^{\times}[0]=gr^0(\underline{\O_F^{\times}})[0]
\] and \[
R\GGamma(B_{W_k},gr^i(\underline{\O_L^{\times}}))=k[0]=gr^i(\underline{\O_F^{\times}})[0].
\]
Consequently, we have \[
R\GGamma(B_{W_k},\underline{\O_L^{\times}/\UU_L^i})\cong \underline{\O_F^{\times}/\UU_F^i}[0]
\] and thus \[
R\GGamma(B_{W_k},\underline{\O_L^{\times}})\cong R\GGamma(B_{W_k},R\lim_{\substack{\leftarrow\\ i\ge 1}} \underline{\O_L^{\times}/\UU_L^i})\cong R\lim_{\substack{\leftarrow\\ i\ge 1}} R\GGamma(B_{W_k},\underline{\O_L^{\times}/\UU_L^i})=R\lim_{\substack{\leftarrow\\ i\ge 1}}\underline{\O_F^{\times}/\UU_F^i}[0]=\underline{\O_F^{\times}}[0].
\]
\end{proof}
\begin{prop}\label{lcohom}
The cohomology of $\hat{W}_F$ with coefficients in $\overline{L}^{\times}$ is given by \[
\HH^q(B_{\hat{W}_F},\overline{L}^{\times})=\begin{array}{ll}
\underline{F^{\times}} & q=0\\
\Z & q=1 \\
0 & q\ge 2,
\end{array}
\] where $F^{\times}$ has its natural topology as a subspace of $F$ and $\Z$ has the discrete topology.
\end{prop}
\begin{proof}
By \cref{vanishing} and by the exact sequence \eqref{eqn:degeneratehs}, we have $\HH^q(B_{\hat{W}_F},\overline{L}^{\times})\cong \HH^q(B_{W_k},\underline{L^{\times}})$. Let us consider the short exact sequence of condensed $W_k$-modules \[
0\rightarrow \underline{\O_L^{\times}}\rightarrow \underline{L^{\times}}\rightarrow \Z \rightarrow 0.
\] The result follows from the long exact cohomology sequence and from \cref{OLcohom}.
\end{proof}
\begin{cor}\label{muncohomology}
Let $n\in \N$. The cohomology of $\hat{W}_F$ with coefficients in $\mu_n(\overline{L}^{\times})$ is given by \[
\HH^q(B_{\hat{W}_F},\mu_n)=\begin{array}{ll}
\mu_n(F) & q=0\\
F^{\times}/(F^{\times})^n & q=1\\
\Z/n\Z & q=2 \\
0 & q\ge 3.
\end{array}\] In particular, the cohomology groups are finite. We set $\mu\coloneqq\lim_{\rightarrow} \mu_n$. Then we have \[
\HH^2(B_{\hat{W}_F},\mu)=\lim_{\rightarrow\, n} \HH^2(B_{\hat{W}_F},\mu_n)=\Q/\Z.
\]
\end{cor}
\begin{proof}
The result follows from \cref{lcohom} and by the long exact cohomology sequence associated to \eqref{exactseq}.
\end{proof}

\begin{defn}\label{defn:twisted} We set $\Z(1)\coloneqq \overline{L}^{\times}[-1]$ and $\R(1)\coloneqq \R[-1]$ in $\DDD(B_{\hat{W}_F})$. Let us consider the shifted valuation morphism $\Z(1)\to\R(1)$. We define \[
\R/\Z(1)\coloneqq \mathrm{cofib}(\Z(1)\to\R(1))=\mathrm{fib}(\overline{L}^{\times}\to\R).
\] \end{defn} We have an exact triangle in $\DDD(B_{\hat{W}_F})$ \begin{equation}\label{twistedtriangle}
\Z(1)\rightarrow \R(1)\rightarrow \R/\Z(1).
\end{equation} \begin{prop}\label{cohomologyrz1}
The cohomology of $\hat{W}_F$ with coefficients in $\R/\Z(1)$ is given by \[
\HH^q(B_{\hat{W}_F},\R/\Z(1))=\begin{array}{ll}
\underline{\O_F^{\times}} & q=0\\
 \R/\Z & q=1,2\\ 
 0 & q\ge 3.
 \end{array} \]
\end{prop}
\begin{proof}
By \cref{banachcoefficients}, we have $\HH^q(B_{\hat{I}},\R)=0$ for all $q>0$. Therefore, we get \[ \HH^q(B_{\hat{W}_F},\R(1))=\HH^{q-1}(B_{W_k},\R)=\begin{array}{ll}
\R & q=1,2\\ 0 & q\neq 1,2.
\end{array}\] 
The long exact cohomology sequence associated to \eqref{twistedtriangle} and \cref{lcohom} imply the result.
\end{proof}

Consequently, there is a trace map 
\begin{equation}\label{tracemap}R\GGamma(B_{\hat{W}_F},\R/\Z(1))\to \R/\Z[-2].
\end{equation} 
If $M\in\DDD^b(B_{\hat{W}_F})$, we set \[
M^D\coloneqq R\underline{\Hom}(M,\R/\Z(1))\in \DDD(B_{\hat{W}_F}).
\] In the next section, we study the condensed abelian groups $\HH^q(B_{\hat{W}_F},M)$ and $\HH^q(B_{\hat{W}_F},M^D)$ for some $\hat{W}_F$-module $M$.

\subsection{Condensed structures on the cohomology groups}
We begin this section by studying some abelian groups of interest. \begin{defn}
Let $A$ be an abelian group. We say that $A$ is \begin{enumerate}[1)]
\item of \emph{finite $\Z$-type} if $A\cong \Z^r\oplus F$ for some $r\in \N$ and some $F$ finite abelian group; equivalently, if $A$ is an extension of a finite power of $\Z$ by a finite abelian group;
\item of \emph{finite $\Q_p/\Z_p$-type} if $A\cong (\Q_p/\Z_p)^r\oplus F$ for some $r\in \N$ and some $F$ finite abelian group; equivalently, if $A$ is an extension of a finite power of $\Q_p/\Z_p$ by a finite abelian group.
\end{enumerate}
\end{defn}

The equivalence in 1) comes from the fact that $\Z^r$ is a projective abelian group, while the equivalence in 2) follows from \begin{lem}\label{divisibleextension}
Let $E$ be an extension of a divisible group $D$ by a finite group $F$. Then $E\cong D\oplus F'$ for $F'$ a quotient of $F$.
\end{lem}
\begin{proof}
This is \cite[Lemma 5.5]{GeisMor}.
\end{proof}
\begin{lem}\label{extsubquo}
\begin{enumerate}[(a)]
\item An extension of an abelian group of finite $\Z$-type (resp.\ of finite $\Q_p/\Z_p$-type) by a finite abelian group is of finite $\Z$-type (resp.\ of finite $\Q_p/\Z_p$-type). \label{extsubquo:ext}
\item Finite $\Z$-type and finite $\Q_p/\Z_p$-type abelian groups are stable by taking subgroups. \label{extsubquo:sub}
\item Finite $\Z$-type and finite $\Q_p/\Z_p$-type abelian groups are stable by taking quotients. \label{extsubquo:quo}
\end{enumerate}
\end{lem}
\begin{proof}
(a) is clear. We only prove (b) and (c) for finite $\Q_p/\Z_p$-type abelian groups. For (b), let $A$ be a subgroup of $(\Q_p/\Z_p)^r\oplus F$, where $n\in\N$ and $F$ finite. By (a) we can reduce to the case where $A$ is a subgroup of $(\Q_p/\Z_p)^r$. Hence $A$ is a torsion $p$-group with a finite $p$-torsion. By \cite[Ch.\ III, Theorem 19.2 and Exercise 19]{Fuchs} we are done.

For (c), let $Q$ be a quotient $(\Q_p/\Z_p)^r\oplus F/A$, for $A$ a subgroup of $(\Q_p/\Z_p)^r\oplus F$. By (a) we can reduce to the case $Q=(\Q_p/\Z_p)^r/A$. Hence $Q$ is a torsion $p$-group with a finite $p$-torsion. By \cite[Ch.\ III, Theorem 19.2 and Exercise 19]{Fuchs} we are done.
\end{proof}

\vspace{0.5 em}

We determine the structure of $R\GGamma(B_{\hat{W}_F},M)$ and $R\GGamma(B_{\hat{W}_F},M^D)$ in some cases of interest. In particular, we consider $M$ either a finite-dimensional real vector spaces with its Euclidean topology or a finite $\Z$-type discrete abelian group. 

Let $V$ be a finite-dimensional real vector space with a continuous action of $W_F/U$, where $U$ is an open normal subgroup of $I$. By \cite[VII.2.1, Prop.\ 1]{BourbakiGT}, the action is $\R$-linear. We set $V^*\coloneqq \Hom_{\R}(V,\R)$. 
 \begin{lem}\label{lem:homrealgm}
We have $R\underline{\Hom}(\R,\overline{L}^{\times})=0$ in $\DDD(B_{\hat{W}_F})$.
\end{lem}
\begin{proof}
By \cref{rmk:localizediso,rmk:morphismofinternalhom}, it is enough to check that $R\underline{\Hom}_{B_{\hat{I}}}(\R,\overline{L}^{\times})=0$. By \cref{internalhomdescr}, (i), it is sufficient to show that for every finite extension $K/L$ of group $G$, we have \[
R\underline{\Hom}_{B_{G}}(\R,R\pi_{K,*}\overline{L}^{\times})=0.
\]
We have \[
(R\pi_{K,*}\overline{L}^{\times})_{|EG}=R\GGamma(B_{\hat{I}}/EG,\overline{L}^{\times}_{|EG})=\lim_{\substack{\rightarrow\\K'/K}} R\GGamma(B_{G_{K'/K}},\underline{(K')^{\times}})=\underline{K^{\times}},
\] where the last equality follows from \cref{vanishing}. Since $\underline{K^{\times}}$ is solid and $\R^{L\blacksquare}=0$ (\cite[Corollary. 6.1,(iii)]{LCM}), we have \[
R\underline{\Hom}_{B_{G}}(\R,R\pi_{K,*}\overline{L}^{\times})_{|EG}=R\underline{\Hom}_{\CC}(\R,\underline{K^{\times}})=R\underline{\Hom}_{\Solid}(\R^{L\blacksquare},\underline{K^{\times}})=0.
\] We conclude by observing that $\CC\to B_G$ is a localisation morphism.
\end{proof}
\begin{rmk}\label{rmk:dualofv}
There is an equivalence $\underline{V^*}\cong R\underline{\Hom}(\underline{V},\R)$ in $\DDD(B_{\hat{W}_F})$. Hence by \cref{lem:homrealgm}, the induced  morphism $\underline{V^*}[-1]\rightarrow \underline{V}^D$ is an equivalence in $\DDD(B_{\hat{W}_F})$. We denote $\underline{V}$ and $\underline{V^*}$ simply by $V$ and $V^*$ from now on.
\end{rmk}
\begin{prop}\label{rmk:structurervs}
Let $V$ be a finite-dimensional real vector space with a continuous action of $W_F/U$, with $U\subset I$ an open subgroup. Then $\HH^q(B_{\hat{W}_F},V)$ is a finite-dimensional vector space for $q=0,1$, and vanishes for all $q\ge 2$. Moreover, $\HH^q(B_{\hat{W}_F},V^D)$ is a finite-dimensional real vector space for $q=1,2$, and vanishes for all $q\neq 1,2$.
\end{prop}
\begin{proof}
By \cref{prop:cont,banachcoefficients}, we have \[ \begin{split}R\GGamma(B_{\hat{W}_F},V)&=R\GGamma(B_{W_k},R\GGamma(B_{\hat{I}},V))\\ &=R\GGamma(B_{W_k},\lim_{\substack{\rightarrow\\ U'\subset U}}\,R\GGamma(B_{I/U'},V))\\ &= R\GGamma(B_{W_k},V^{I/U}). \end{split}\] By \cref{cohomologyz}, this complex is concentrated in cohomological degrees $0,1$. Moreover, its cohomology groups are finite dimensional real vector spaces. The same holds if we replace $V$ with $V^*$. We conclude by \cref{rmk:dualofv}. \end{proof}

Let $M$ be a discrete abelian group with a continuous action of $W_F$ (resp.\ $I$). Let us consider $M$ as an object of $\Ab(B_{\hat{W}_F})$ (resp.\ $\Ab(B_{\hat{I}})$) by pullback along $B_{\hat{W}_F}\rightarrow B_{\hat{W}_F}(\Set)=B_{W_F}(\Set)$ (resp.\ $B_{\hat{I}}\rightarrow B_{\hat{I}}(\Set)=B_I(\Set))$.
\begin{rmk}\label{rmk:cohomdiscretemod}
Let $N$ be a discrete abelian group with a continuous action of $I$. By \cref{prop:cont} we have \[
R\GGamma(B_{\hat{I}},N)=\underset{\substack{\rightarrow\\ U}}{\lim}\,R\GGamma(B_{I/U},N^U).
\] By \cref{prop:discretecohom,prop:discretegroupsections}, this is a complex of discrete abelian groups and it is equivalent to  $\underline{R\GGamma(B_{I}(\Set),N)}$ in $\DDD(\Ab(\CC))$. 
The same holds for $\hat{W}_F$ and $W_F$ instead of $\hat{I}$ and $I$ respectively. \end{rmk}
\begin{prop}\label{weilcohomologydiscrete}
 Let $M$ be a discrete abelian group with a continuous action of $W_F$. Then $\HH^q(B_{\hat{W}_F},M)$ is discrete for all $q$ and vanishes for all $q\ge 4$. If $M$ is torsion, then we have $\HH^3(B_{\hat{W}_F},M)=0$.
 \end{prop}
 \begin{proof}
By \cref{rmk:cohomdiscretemod}, the condensed abelian group $\HH^q(B_{\hat{W}_F},M)$ is discrete for all $q$.
 By \cite[Ch. II, \S 3.3 c)]{serreGC}, $\H^q_{cont}(I,M)=0$ for all $q\ge 3$, and even for all $q\ge 2$ if $M$ is torsion. We conclude again by \cref{rmk:cohomdiscretemod} and by \eqref{eqn:degeneratehs}.
 \end{proof}
Suppose that $M$ is finite. Then there exists an open normal subgroup $U$ of $I$ acting trivially on $M$. We set $H\coloneqq W_F/U$. \begin{lem}\label{lem:finiterealhomvanishing}
Let $M$ be a finite abelian group with a continuous action of $W_F$. Then we have $R\underline{\Hom}(M,\R)=0$ in $\DDD(B_{\hat{W}_F})$.
\end{lem}
\begin{proof}
By \cref{rmk:localizediso,rmk:morphismofinternalhom} it is enough to prove that $R\underline{\Hom}_{B_{\hat{U}}}(M|_{|EH},\R_{|EH})=0$ in $\DDD(B_{\hat{U}})$. Since $M_{|EH}\cong \Z/n_1\Z\oplus \dots \oplus \Z/n_k\Z$, we can reduce to $M=\Z/n\Z$. We have the exact triangle in $\DDD(B_{\hat{U}})$ \[
R\underline{\Hom}(\Z/n\Z,\R)\longrightarrow R\underline{\Hom}(\Z,\R)\overset{\cdot n}{\longrightarrow} R\underline{\Hom}(\Z,\R).
\] Since $\cdot n:\R\to\R$ is an isomorphism, the result follows.
\end{proof}
\begin{lem}\label{dualizingobjectfinite}
Let $M$ be a finite abelian group with a continuous action of $W_F$. Then $M^D$ is represented by the finite abelian group $M'\coloneqq \Hom(M,\mu)=R\Hom(M,\mu)$.\end{lem}
\begin{proof}
By \cref{lem:finiterealhomvanishing}, it is enough to check that we have an equivalence \[
R\Hom(M,\mu)\rightarrow R\underline{\Hom}(M,\overline{L}^{\times})
\] in $\DDD(B_{\hat{W}_F})$. The inclusion $\mu\subset \overline{L}^{\times,\delta}$ yields a morphism $\Hom(M,\mu)\to R\Hom(M,\overline{L}^{\times,\delta})$ in $\DDD(B_{W_F}(\Set))$. By \cref{rmk:morphismofinternalhom}, the morphism of topoi $\DDD(B_{\hat{W}_F})\to \DDD(B_{W_F}(\Set))$ yields a morphism \[R\Hom_{\DDD(B_{W_F}(\Set))}(M,\overline{L}^{\times,\delta})\rightarrow R\underline{\Hom}_{\DDD(B_{\hat{W}_F})}(M,\overline{L}^{\times,\delta})\rightarrow R\underline{\Hom}_{\DDD(B_{\hat{W}_F})}(M,\overline{L}^{\times}). \] We get a morphism $R\Hom(M,\mu)\to R\underline{\Hom}(M,\overline{L}^{\times})$ in $\DDD(B_{\hat{W}_F})$. By \cref{rmk:localizediso}, it is enough to check that the morphism \[
R\Hom(M_{|EH},\mu_{|EH})\to R\underline{\Hom}(M_{|EH},\overline{L}^{\times}_{|EH})
\] is an equivalence in $\DDD(B_{\hat{U}})$. We can suppose $M_{|EH}=\Z/n\Z$. The result follows since $R\underline{\Hom}(\Z/n\Z,\overline{L}^{\times})=\mu_n$ is discrete.
\end{proof}
\begin{prop}\label{finitecoeffstructure}
Let $M$ be a finite abelian group with a continuous action of $W_F$. Then $\HH^q(B_{\hat{W}_F},M)$ and $\HH^q(B_{\hat{W}_F},M^D)$ are finite for all $q$, and vanish for $q\ge 3$.
\end{prop}
\begin{proof}
If $M=\mu_n$, this is just a consequence of \cref{muncohomology}. Then we adapt the proof of the finiteness statement in \cite[Theorem 2.1]{ADT}, plus \cref{weilcohomologydiscrete}. By \cref{dualizingobjectfinite}, the result for $M^D$ follows as well.
\end{proof}
 \begin{lem}\label{mconcentration}
Let $M$ be a free abelian group of finite $\Z$-type with a continuous action of $W_F$. Then $R\GGamma(B_{\hat{W}_F},M)$ is concentrated in degrees $0,1,2$.
\end{lem}
\begin{proof}
By \cref{weilcohomologydiscrete}, it is enough to show that $\HH^3(B_{\hat{W}_F},M)=\underline{\H^3(B_{W_F}(\Set),M)}=0$. This is \cite[Theorem 3.2.1]{Karpuk}.\end{proof}
 \begin{lem}
\label{wcohomologyofz}
Let $F$ be a finite extension of $\Q_p$. Then we have \[
\HH^q(B_{\hat{W}_F},\Z)= \begin{array}{ll}
\Z & q=0,1,\\
J^{\vee} & q=2,\\
0 & q\ge 3, \end{array}
\] where $J$ is the kernel of $G_K^{ab}\twoheadrightarrow G_k$.
\end{lem}
\begin{proof}
By \cref{rmk:cohomdiscretemod,mconcentration}, we have $\HH^q(B_{\hat{W}_F},\Z)=\underline{\H^q(B_{W_F}(\Set),\Z)}$ for all $q$ and $\HH^q(B_{\hat{W}_F},\Z)=0$ for all $q\ge 3$. In particular, $\H^0(B_{\hat{W_F}},\Z)=\Z$. In order to determine $\H^1(B_{W_F}(\Set),\Z)$, we observe that  $\H^1(B_I(\Set),\Z)=\Hom^{cont}(I,\Z)=0$. Thus we have \[
\H^1(B_{W_F}(\Set),\Z)\cong\H^1(B_{W_k}(\Set),\Z)\cong \Z.\] It remains to determine $\H^2(B_{W_F}(\Set),\Z)$.
Let us consider the long exact cohomology sequence associated to \[
0\rightarrow \Z\rightarrow \Q\rightarrow\Q/\Z\rightarrow 0.
\] 
Since $\H^q(B_I(\Set),\Q)=0$ for all $q\ge 1$, we have\[ \qquad \H^1(B_{W_F}(\Set),\Q)\cong \H^1(B_{W_k}(\Set),\Q)\cong \Q.
\]

By the Snake Lemma, we have
\[
\begin{split}
\H^2(B_{W_F}(\Set),\Z)&\cong\mathrm{coker}(\H^1(B_{W_F}(\Set),\Q)/\H^1(B_{W_F}(\Set),\Z)\rightarrow \H^1(B_{W_F}(\Set),\Q/\Z))\\
&\cong \mathrm{coker}(\H^1(B_{W_k}(\Set),\Q)/\H^1(B_{W_k}(\Set),\Z)\rightarrow \H^1(B_{W_F}(\Set),\Q/\Z))\\
&\cong \mathrm{coker}(\H^1(B_{W_k}(\Set),\Q/\Z)\rightarrow \H^1(B_{W_F}(\Set),\Q/\Z)),
\end{split}
\] where the last isomorphism follows from the fact that $\H^2(B_{W_k}(\Set),\Z)=0$. 
Since $\Q/\Z$ is torsion, by \cite[Proposition 4.1.1.]{Karpuk} we have \[ 
\H^2(B_{W_F}(\Set),\Z)=\mathrm{coker}(\H^1(B_{G_k}(\Set),\Q/\Z)\rightarrow \H^1(G_F(\Set),\Q/\Z))=\mathrm{coker}((G_k)^{\vee}\to (G_F^{ab})^{\vee})= J^{\vee}.
\]
\end{proof}
\begin{rmk}
By Local Class Field Theory, we have an isomorphism of topological groups \[
\O_F^{\times}\overset{\sim}{\longrightarrow} J
\] Hence $R\GGamma(B_{\hat{W}_F},\Z)$ and $R\underline{\Hom}(R\GGamma(B_{\hat{W}_F},\R/\Z(1)),\R/\Z[-2])$ have the same cohomology.
\end{rmk}

 Let $F'$ be a finite extension of $F$, and let $G\coloneqq \Gal(F'/F)$.
 \begin{thm}\label{thm:structurem}
Let $M$ be a free abelian group of finite $\Z$-type with a continuous action of $G$. Then $\HH^q(B_{\hat{W}_F},M)$ is discrete of finite $\Z$-type for $q=0,1$, discrete of finite $\Q_p/\Z_p$-type for $q=2$, and vanishes for all $q\ge 3$.
\end{thm}
\begin{proof}
By \cref{rmk:cohomdiscretemod}, we have $\HH^q(B_{\hat{W}_F},M)=\underline{\H^q(B_{W_F}(\Set),M)}$. Firstly, the abelian group \[
\H^0(B_{W_F}(\Set),M)=M^{G}
\] is of finite $\Z$-type. Moreover, the exact sequence \[
0\rightarrow \H^1(B_{W_k}(\Set),M^I)\rightarrow \H^1(B_{W_F}(\Set),M)\rightarrow \H^0(B_{W_k}(\Set),\H^1(B_I(\Set),M))\rightarrow 0
\] presents $\H^1(B_{W_F}(\Set),M)$ as an extension of a finite group by a finite $\Z$-type abelian group. Hence $\H^1(B_{W_F}(\Set),M)$ is of finite $\Z$-type. It remains to determine the structure of $\H^2(B_{W_F}(\Set),M)$. Let us consider the Hochschild-Serre spectral sequence \[
E_2^{i,j}=\H^i(B_G(\Set),\H^j(B_{W_{F'}}(\Set),M))\implies \H^{i+j}(B_{W_F}(\Set),M).
\] By \cref{wcohomologyofz}, $E_2^{i,0}$ and $E_2^{i,1}$ are finite for all $i\ge 1$. Moreover, by \cref{wcohomologyofz,extsubquo}, $E_2^{0,2}$ is of finite $\Q_p/\Z_p$-type.
Thus we have a finite filtration \[
0\subset F^2\subset F^1\subset F^0=\H^2(B_{W_F}(\Set),M), \qquad F^0/F^1=E_{\infty}^{0,2},\, F^1/F^2=E_{\infty}^{1,1},\, F^2=E_{\infty}^{2,0},
\] where $E_{\infty}^{1,1}$ and $E_{\infty}^{2,0}$ are finite and $E_{\infty}^{0,2}$ is of finite $\Q_p/\Z_p$-type. Hence $\H^2(B_{W_F}(\Set),M)$ is of finite $\Q_p/\Z_p$-type by \cref{extsubquo}, (a).
\end{proof}
\begin{thm}\label{thm:structuremd}
Let $M$ be a free abelian group of finite $\Z$-type with a continuous action of $G$. Then the Pontryagin dual of $\HH^q(B_{\hat{W}_F},M^D)$ is discrete of finite $\Z$-type for $q=1,2$, discrete of finite $\Q_p/\Z_p$-type for $q=0$, and vanishes for all $q\ge 3$.
\end{thm}
\begin{proof}
Let us consider the Hochschild-Serre spectral sequence \begin{equation}\label{eqn:lem:concentration:hs}
E_2^{i,j}=\HH^i(B_{G},\HH^j(B_{\hat{W}_{F'}},M^D))\implies \HH^{i+j}(B_{\hat{W}_F},M^D).
\end{equation}  We determine the structure of $E_2^{i,j}$ and $E_{\infty}^{i,j}$ for all $i,j$. By \cref{cohomologyrz1,extsubquo}, $E_2^{0,0}$ is the Pontryagin dual of a finite $\Q_p/\Z_p$-type abelian group. Moreover, $E_2^{0,j}$ is the Pontryagin dual of a finite $\Z$-type abelian group for $j=1,2$. Finally, $E_2^{0,j}$ vanishes for $j\ge 3$. Consequently, $E_{\infty}^{0,j}$ is the Pontryagin dual of a finite $\Z$-type abelian group for $j=1,2$ and vanishes for all $j$.

Let $i\ge 1$. By \cref{cohomologyrz1}, there exist $N\in \N$ and a finite abelian group $K$ such that $E_2^{i,0}=\H^i(B_G,\underline{\Z_p^N\oplus K})$. Hence $E_2^{i,0}$ is represented by a compact Hausdorff topological abelian group (see \cref{finitegrouptop}, (ii)). Up to replacing $K$ with its $p^{\infty}$-torsion, the abelian group \[\HH^i(B_G,\underline{\Z_p^N\oplus K})(*)=\HH^i(B_G(\Set),\Z_p^{N,\delta}\oplus K)\] is a torsion $\Z_p$-module. Moreover, by \cref{cohomologyrz1,banachcoefficients} there exist $n_1,n_2\in \N$ such that $E_2^{i,j}=\HH^{i+1}(B_G,\Z^{n_j})$, for $j=1,2$. Hence $E_2^{i,j}$ is finite for all $i\ge 1$ and for all $j$. Consequently, $E_{\infty}^{i,j}$ is finite for all $i\ge 1$ and for $j=0,1,2$, and vanishes for all $j\ge 3$.

We now determine the structure of $\HH^q(B_{\hat{W}_F},M^D)$. For all $q\ge 2$, \eqref{eqn:lem:concentration:hs} gives a three-terms filtration \begin{equation}\label{eqn:lem:concentration:fil}
0\subset F^2 \subset F^1\subset F^0=\HH^q(B_{\hat{W}_F},M^D), \qquad gr^i\coloneqq F^i/F^{i+1}=E_{\infty}^{q-i,i}.
\end{equation}
Let $q\ge 3$. Then all the graded groups in \eqref{eqn:lem:concentration:fil} are finite. Hence $\HH^q(B_{\hat{W}_F},M^D)$ is finite for all $q\ge 3$. In particular, we have $\HH^q(B_{\hat{W}_F},M^D)=\lim_m \HH^q(B_{\hat{W}_F},M^D)/m$. By the inclusion $\lim_m \HH^q(B_{\hat{W}_F},M^D)/m\hookrightarrow \lim_m \HH^q(B_{\hat{W}_F},(M/m)^D)$ and by \cref{finitecoeffstructure}, we have $\lim_m \HH^q(B_{\hat{W}_F},M^D)/m=0$ for all $q\ge 3$. Hence $\HH^q(B_{\hat{W}_F},M^D)$ vanishes for all $q\ge 3$.

 Let $q=2$. Then $gr^i$ is finite for $i=0,1$, and $gr^0$ is the Pontryagin dual of a finite $\Z$-type abelian group. Hence $\HH^2(B_{\hat{W}_F},M^D)^{\vee}$ is of finite $\Z$-type. 

Let $q=1$. Then \eqref{eqn:lem:concentration:hs} gives an exact sequence \[
0\rightarrow E_{\infty}^{1,0}\rightarrow \HH^1(B_{\hat{W}_F},M^D)\rightarrow E_{\infty}^{0,1}\rightarrow 0.
\] The left term is finite and the right term is the Pontryagin dual of a finite $\Z$-type abelian group. Hence $\HH^1(B_{\hat{W}_F},M^D)^{\vee}$ is of finite $\Z$-type.

Let $q=0$. By \eqref{eqn:lem:concentration:hs}, we have $\HH^0(B_{\hat{W}_F},M^D)^{\vee}=(E_2^{0,0})^{\vee}$, which is of finite $\Q_p/\Z_p$-type by \cref{cohomologyrz1,extsubquo}, (c). 
\end{proof}
\cref{rmk:structurervs,thm:structurem} imply the following.
\begin{prop}\label{prop:structurecircle}
Let $M$ be a finite power of $\R/\Z$ with a continuous action of $G$. Then there exist $n_0,n_1,m_1\in \N$ and finite abelian groups $H_0,H_1$ such that \[
\HH^q(B_{\hat{W}_F},M)\cong \begin{array}{ll}
(\R/\Z)^{n_0}\oplus H_0 & q=0,\\
(\R/\Z)^{n_1}\oplus (\Q_p/\Z_p)^{m_1}\oplus H_1 & q=1,\\
0 & q\ge 2. \end{array}
\]
\end{prop}
\begin{proof}
By \cref{lowercohomologygroups}, 1), the condensed abelian group $\HH^0(B_{\hat{W}_F},M)$ is represented by a closed subgroup of $(\R/\Z)^n$, hence it is of the desired form. Let us consider the exact sequence in $\Ab(B_{\hat{W}_F})$
\begin{equation}\label{prop:structurecircle:exactsequence}
0\rightarrow L\rightarrow R\rightarrow M\rightarrow 0,
\end{equation} where $R\coloneqq \underline{\Hom}(\R,M)$ and $L\coloneqq\underline{\Hom}(\R/\Z,M)$. Applying \cref{rmk:structurervs,thm:structurem} to $R$ and $L$ respectively, we get $\HH^q(B_{\hat{W}_F},M)=0$ for all $q\ge 2$.

We only need to determine $\HH^1(B_{\hat{W}_F},M)$. If the action is trivial, by the long exact cohomology sequence associated to \eqref{prop:structurecircle:exactsequence}, we get \[
\HH^1(B_{\hat{W}_F},M)\cong (\R/\Z)^n\oplus ((\O_F^{\times})^{\vee})^n.
\] In the general case, we have an exact sequence \[
0\rightarrow \HH^1(B_G,\HH^0(B_{\hat{W}_{F'}},M))\rightarrow \HH^1(B_{\hat{W}_F},M)\rightarrow K\rightarrow 0,
\] where $K$ is represented by a closed subgroup of $\HH^1(B_{\hat{W}_{F'}},M)\cong (\R/\Z)^n\oplus ((\O_{F'})^{\vee})^n$. Hence $K$ is a direct sum of a finite power of $\R/\Z$ and a discrete abelian group of finite $\Q_p/\Z_p$-type, while $\HH^1(B_G,\HH^0(B_{\hat{W}_{F'}},M))\cong \HH^2(B_G,\Z^n)$ is a finite abelian group. We conclude by \cref{divisibleextension}.
\end{proof}

\section{Duality}\label{section3}
By \cref{cohomologyrz1} we have a trace map \[
R\GGamma(B_{\hat{W}_F},\R/\Z(1))\rightarrow \R/\Z[-2].
\] Let $M\in \DDD^b(B_{\hat{W}_F})$. We have a cup-product pairing in $\DDD(\Cond(\Ab))$ \begin{equation}\label{cupprod}
R\GGamma(W_F,M)\otimes^L R\GGamma(W_F,M^D)\rightarrow \R/\Z[-2]
\end{equation} inducing maps \begin{equation}\label{cupprodmap}\begin{split}
\psi(M^D)&:R\GGamma(W_F,M)\rightarrow R\underline{\Hom}(R\GGamma(W_F,M^D),\R/\Z[-2])\\
\psi(M)&:R\GGamma(W_F,M^D)\rightarrow R\underline{\Hom}(R\GGamma(W_F,M),\R/\Z[-2]).
\end{split}
\end{equation}
In this section we first define a full stable $\infty$-subcategory $\DDD^{perf}_{\Z,\R}$ of $\DDD^b(\Cond(\Ab))$ (see \cref{locallycompactabeliangroups}). Then we prove that for all $M\in \DDD^{perf}_{\Z,\R}$ with an action of a finite quotient of $G_F$, \eqref{cupprod} is perfect, i.e.\ both $\psi(M)$ and $\psi(M^D)$ are equivalences (see \cref{mainresult}). In order to do this, we need two tools: these are presented in \cref{dualityforwk,conservativityofcompletion}.
\subsection{Locally compact abelian groups of finite ranks}\label{locallycompactabeliangroups}
In the following, we recall the definitions of the derived $\infty$-categories $\DDD^b(\FLCA)$ and $\DDD^b(\Ab(\CC))$ (see \cite[\S 2.1]{GeisserMor2}). We show that $\FLCA$ is stable by extensions in $\Ab(\CC)$. Moreover, we define a convenient full stable $\infty$-subcategory of $\DDD(\Ab(\CC))$.

Let $A$ be a quasi-abelian category in the sense of \cite{Schn}. Let $\NN\subset C^b(A)$ be the full subcategory of strictly acyclic complexes, and let $S$ be the set of strict quasi-isomorphisms. We define the bounded derived $\infty$-category of $A$ \[
\DDD^b(A)\coloneqq N_{dg}(C^b(A))/N_{dg}(\NN)\coloneqq N_{dg}(C^b(A))[S^{-1}],
\] where $N_{dg}(-)$ denotes the differential graded nerve \cite[Construction 1.3.1.6]{HA}. The homotopy category \[
D^b(A)\coloneqq h(\DDD^b(A))\cong h(N_{dg}(C^b(A)))/h(N_{dg}(\NN))
\] is equivalent to the classical Verdier quotient. Hence it is the bounded derived category of the quasi-abelian category $A$ in the sense of \cite{Schn}.
\vspace{0.5 em}

The category $\Ab(\CC)$ is an abelian category \cite[Theorem 2.2]{LCM}. We define $\DDD^b(\Ab(\CC))$ as above, where $\NN\subset C^b(\Ab(\CC))$ is the full subcategory of acyclic complexes. This is the bounded derived $\infty$-category of condensed abelian groups, and its homotopy category $D^b(\Ab(\CC))$ appears in \cite{LCM}.

We denote by $\FLCA$ the category of locally compact abelian groups of finite ranks in the sense of \cite[Definition 2.6]{Hoff}, which is a quasi-abelian category (see \cite[Corollary 2.11]{Hoff}). Let $\DDD^b(\FLCA)$ be its bounded derived $\infty$-category. Following \cite{GeisserMor2}, we observe that $\DDD^b(\FLCA)$ is a stable $\infty$-category in the sense of \cite{HA}. Its homotopy category is the bounded derived category $D^b(\FLCA)$ defined in \cite{Hoff}. In \cite{GeisserMor2} the authors define an internal Hom \[
R\underline{\Hom}(-,-):\DDD^b(\FLCA)^{op}\times \DDD^b(\FLCA)\rightarrow \DDD^b(\FLCA).
\]
The Pontryagin dual $X^{\vee}$ is given by \[
R\underline{\Hom}(-,\R/\Z):\DDD^b(\FLCA)^{op}\rightarrow \DDD^b(\FLCA), \qquad X\mapsto X^{\vee}.
\]
The functor $\underline{(-)}:\FLCA \rightarrow \Ab(\CC)$ sending $A$ to the associated condensed abelian group $\underline{A}$ is fully faithful (see \cite[Proposition 1.7]{LCM}) and sends strict quasi-isomorphisms to usual quasi-isomorphisms. Thus it induces a functor \[
\DDD^b(\FLCA)\rightarrow \DDD^b(\Ab(\CC)).
\]
\begin{rmk}\label{internalrhom}
Considering $\underline{\Hom}$ instead of $\Hom$ in the proof of \cite[Corollary 4.9]{LCM}, it follows that the two functors \[
R\underline{\Hom}(-,-):\DDD^b(\FLCA)^{op}\times \DDD^b(\FLCA)\longrightarrow \DDD^b(\FLCA)\overset{\underline{(-)}}{\longrightarrow} \DDD^b(\Ab(\CC))
\] and \[
R\underline{\Hom}(-,-):\DDD^b(\FLCA)^{op}\times \DDD^b(\FLCA)\overset{\underline{(-)}\times\underline{(-)}}{\longrightarrow}\DDD^b(\Ab(\CC))^{op}\times \DDD^b(\Ab(\CC))\rightarrow \DDD^b(\Ab(\CC))
\] coincide.
\end{rmk}
\begin{lem}\label{lem:exff}
The functor $\DDD^b(\FLCA)\rightarrow \DDD^b(\Ab(\CC))$ is an exact and fully faithful functor of stable $\infty$-categories.
\end{lem}
\begin{proof}
We observe that \[
N_{dg}(C^b(\FLCA))\rightarrow N_{dg}(C^b(\Ab(\CC)))\rightarrow \DDD^b(\Ab(\CC))
\] is exact, and sends strictly acyclic complexes of $C^b(\FLCA)$ to $0$. By \cite[Theorem I.3.3,(i)]{TCH}, it induces an exact functor\footnote{If $\kappa$ is an uncountable strong limit cardinal such that $\mathrm{cof}(\kappa)>\aleph_0$, then the cardinality of any $A\in\FLCA$ is less than $\kappa$. Hence the category $\FLCA$ is essentially small, and so is the $\infty$-category $N_{dg}(C^b(\FLCA))$. Consequently, \cite[Theorem I.3.3,(i)]{TCH} applies.} $\DDD^b(\FLCA)\rightarrow \DDD^b(\Ab(\CC))$. Moreover, the induced functor on homotopy categories \[
D^b(\FLCA)\rightarrow D^b(\Ab(\CC))
\] is fully faithful by \cite[Lemma 2.1]{GeisserMor2} and \cite[Corollary 4.9]{LCM}. Hence $\DDD^b(\FLCA)\to \DDD^b(\Ab(\CC))$ is an exact functor of stable $\infty$-categories which induces a fully faithful functor between the corresponding homotopy categories. The result follows.
\end{proof}

The stable $\infty$-category $\DDD^b(\FLCA)$ is endowed with a $t$-structure by \cite[Section 1.2.2]{Schn}, since a $t$-structure on a stable $\infty$-category is defined as a $t$-structure on its homotopy category \cite[Definition 1.2.1.4]{HA}. We denote their heart by $\LH(\FLCA)$.

\begin{lem}\label{lem:tex}
The fully faithful functor $\DDD^b(\FLCA)\rightarrow \DDD^b(\Ab(\CC))$ is $t$-exact. Therefore, the induced functor \[
\LH(\FLCA)\hookrightarrow \Ab(\CC)
\] is exact and fully faithful.
\end{lem}
\begin{proof}
Let $X^{\bullet}\in \DDD^b(\FLCA)^{\ge 0}$. By \cite[Proposition 1.2.19]{Schn}, in $\DDD^b(\FLCA)$ the complex $X^{\bullet}$ is isomorphic to \[
0\rightarrow \mathrm{Coim} (d_X^{-1})\rightarrow X^0\rightarrow X^1\rightarrow \dots,
\] with $\mathrm{Coim} (d_X^{-1})$ in degree -1. The functor $\underline{(-)}:\LCA\rightarrow \Ab(\CC)$ respects cokernels of closed immersions and all kernels, hence it respect coimages.  By \cite[Corollary 2.11]{Hoff}, the same holds for the functor $\underline{(-)}:\FLCA\rightarrow \Ab(\CC)$. Hence in $\DDD^b(\Ab(\CC))$ the complex $\underline{X^{\bullet}}$ is isomorphic to \[
0\rightarrow\mathrm{Coim}(\underline{d_X^{-1}})\rightarrow \underline{X^0}\rightarrow \underline{X^1}\rightarrow \dots,
\] with $\mathrm{Coim}(\underline{d_X^{-1}})$ in degree -1. This complex is acyclic in strictly negative degree. Consequently, we have $\underline{X^{\bullet}}\in \DDD^b(\Ab(\CC))^{\ge 0}$.

Now, let $X^{\bullet}\in \DDD^b(\LCA)^{\le 0}$. By \cite[Proposition 1.2.19]{Schn}, it is isomorphic to \[
\dots \rightarrow X^{-2}\rightarrow X^{-1}\rightarrow \mathrm{Ker}(d_X^0) \rightarrow 0,
\] with $\mathrm{Ker}(d_X^0)$ in degree $0$. Then the complex $\underline{X^{\bullet}}$ is given by \[
\dots \rightarrow \underline{X^{-2}}\rightarrow \underline{X^{-1}}\rightarrow \mathrm{Ker}(\underline{d_X^0}) \rightarrow 0,
\] with $\mathrm{Ker}(\underline{d_X^0})$ in degree $0$, which is acyclic in strictly positive degree. Consequently, we have $\underline{X^{\bullet}}\in\DDD^b(\Ab(\CC))^{\le 0}$.

Hence the functor $\DDD^b(\FLCA)\rightarrow \DDD^b(\Ab(\CC))$ induces a functor on left hearts. Since $\Hom$-sets are computed as in the derived category, the functor $\LH(\FLCA)\rightarrow \Ab(\CC)$ is fully faithful by \cref{lem:exff}.
\end{proof}
\begin{cor}\label{cor:intersections}
Let $A\in \Ab(\CC)$. If $A$ lies in the essential image of $\DDD^b(\FLCA)$, then it lies in the essential image of $\LH(\FLCA)$. In other words, we have \[
\Ab(\CC)\cap \DDD^b(\FLCA)=\LH(\FLCA).
\]
\end{cor}
\begin{proof}
We write $A=\underline{X^{\bullet}}$, with $X^{\bullet}\in \DDD^b(\FLCA)$. The condition $X^{\bullet}\in \DDD^b(\FLCA)^{\le 0}$ is equivalent to the condition \[
\Hom_{\DDD^b(\FLCA)}(X^{\bullet},Y^{\bullet})=0 \qquad \forall Y^{\bullet}\in \DDD^b(\FLCA)^{\ge 1}.
\] Since the functor $\underline{(-)}:\DDD^b(\FLCA)\to \DDD^b(\Ab(\CC))$ is fully faithful and $t$-exact, one gets \[
\Hom_{\DDD^b(\FLCA)}(X^{\bullet},Y^{\bullet})=\Hom_{\DDD^b(\Ab(\CC))}(A,\underline{Y^{\bullet}})=0.
\] Thus $X^{\bullet}\in \DDD^b(\FLCA)^{\le 0}$. Similarly, one gets $X^{\bullet}\in \DDD^b(\FLCA)^{\ge 0}$. The result follows.
\end{proof}
\begin{prop}\label{prop:stableextensions}
Up to identification with its essential image via $\underline{(-)}:\FLCA\to \Ab(\CC)$, the category $\FLCA$ is a full subcategory of $\Ab(\CC)$ which is stable by extensions. \end{prop}
\begin{proof}
Let us consider the exact sequence in $\Ab(\CC)$ \[
0\rightarrow \underline{A'}\rightarrow E \rightarrow \underline{A''}\rightarrow 0,
\] where $A',A''\in \FLCA$. Then we have \[
E=\mathrm{fib}(\underline{A''}\rightarrow \underline{A'}[1])=\underline{\mathrm{fib}(A''\rightarrow A'[1])},
\] where the last equality comes from exactness of $\DDD^b(\FLCA)\rightarrow \DDD^b(\Ab(\CC))$ (see \cref{lem:exff}). In particular, we have $E\in \Ab(\CC)\cap \DDD^b(\FLCA)$. By \cref{cor:intersections}, we have $E\in \LH(\FLCA)$. Since $\FLCA\subset \LH(\FLCA)$ is stable by extensions \cite[Proposition 1.2.29, (c)]{Schn}, the result follows.
\end{proof}
\begin{rmk}
Replacing $\FLCA$ with $\LCA_{\kappa}$, we can build the derived $\infty$-category $\DDD^b(\LCA_{\kappa})$. All the previous results hold for $\DDD^b(\LCA_{\kappa})$ instead of $\DDD^b(\FLCA)$.
\end{rmk}

\vspace{0.5 em}
Let us define a convenient full stable $\infty$-subcategory of $\DDD^b(\Ab(\CC))$.
\begin{defn}\label{def:dperf} We define $\DDD^{perf}_{\Z,\R}$ as the full $\infty$-subcategory of $\DDD^b(\FLCA)$ whose objects are those $X\in \DDD^b(\FLCA)$ such that $R\underline{\Hom}(\R/\Z,X)\in \DDD^{perf}(\Z)$.
\end{defn}
Let $X\in \DDD(\Ab(\CC))$. We have a distinguished triangle in $\DDD(\Ab(\CC))$  \begin{equation}\label{eqn:trianglefordperf}
R\underline{\Hom}(\R/\Z,X)\rightarrow R\underline{\Hom}(\R,X)\rightarrow X.
\end{equation}
If $X\in \DDD^{perf}_{\Z,\R}$, then the term on the left is in $\DDD^{perf}(\Z)$ and the middle term is in $\DDD^{perf}(\R)$.
\begin{rmk} We highlight some properties of $\DDD^{perf}_{\Z,\R}$.
\begin{itemize}
\item $\DDD^{perf}_{\Z,\R}$ is a stable $\infty$-subcategory of $\DDD(\Ab(\CC))$.
\item $\DDD^{perf}(\Z)$ is stable under retracts, hence so is $\DDD^{perf}_{\Z,\R}$.
\item $X\in \DDD^{perf}_{\Z,\R}$ if and only if its Pontryagin dual is in $\DDD^{perf}_{\Z,\R}$. Indeed we have\[R\underline{\Hom}(\R/\Z,X^{\vee})=R\underline{\Hom}(X,\R/\Z^{\vee})=R\underline{\Hom}(R\underline{\Hom}(\R/\Z,X),\Z)[1],\] where the last equality comes from \eqref{eqn:trianglefordperf}.
\item $\R$ and $\Z$ are in $\DDD^{perf}_{\Z,\R}$.
\item If $\mathcal{A}$ is a stable $\infty$-subcategory of $\DDD(\Ab(\CC))$ containing $\Z$ and $\R$, then $\DDD^{perf}_{\Z,\R}$ is contained in $\mathcal{A}$. This follows from stability of $\mathcal{A}$ and \eqref{eqn:trianglefordperf}.
\end{itemize}
\end{rmk}
\begin{prop}\label{prop:degree0}
Let $A\in \DDD^{perf}_{\Z,\R}\cap \FLCA$ with a continuous action of a finite quotient $G$ of $G_F$. Then we have a $G$-equivariant filtration in $\Ab(\CC)$\[
0\subset F^0A\subset F^1A\subset F^2A=A,
\] where \begin{itemize}\item $F^0A\eqqcolon A_{\SSS^1}$ is a power of $\R/\Z$;\\
\item $F^1A/F^0A\eqqcolon A_{\R}$ is a finite-dimensional real vector space;\\
\item $A/F^1A\eqqcolon A_{\Z}$ is a finite $\Z$-type abelian group. \end{itemize}
\end{prop}
\begin{proof}
By the distinguished triangle \eqref{eqn:trianglefordperf}, we get an exact sequence in $\Ab(\CC)$ \[
0\rightarrow \Z^{n_1}\rightarrow \R^{n_2}\rightarrow A \rightarrow \Z^{n_3}\oplus H\rightarrow 0,
\] for some $n_1,n_2,n_3\in \N$ and some finite abelian group $H$. The image of $\Z^{n_1}$ is of the form $\Z v_1\oplus \dots \Z v_k$, for $\{v_1,\dots, v_k\}$ linearly independent vectors of $\R^m$. Hence $\R^{n_2}/\Z^{n_1}\cong (\R/\Z)^k \oplus \R^{n_2-k}$. We have a short exact sequence in $\Ab(\CC)$ \[
0\rightarrow (\R/\Z)^k\oplus \R^{n_2-k}\rightarrow A\rightarrow \Z^{n_3}\oplus H\rightarrow 0.
\] Since we have $A\in \FLCA$, the morphism $(\R/\Z)^k\oplus \R^{n_2-k}\rightarrow A$ is represented by a continuous homomorphism of topological abelian groups.  The image of this morphism is the connected component of $0$, which is a $G$-equivariant subgroup. We call it $F^1A$, and we observe that $A_{\Z}\coloneqq A/F^1A$ is a finite $\Z$-type abelian group.

Moreover, we have an exact sequence in $\Ab(\CC)$ \[
0\rightarrow (\R/\Z)^k\rightarrow F^1A\rightarrow \R^{n_2-k}\rightarrow 0.
\] The image of $(\R/\Z)^k$ is the maximal compact subgroup of $F^1A$, which is a $G$-equivariant subgroup. We call it $F^0A$ or $A_{\SSS^1}$. Then $A_{\R}\coloneqq F^1A/F^0A$ is a finite-dimensional real vector space.
\end{proof}
\begin{prop}\label{prop:degree0cohomology}
Let $A\in\DDD^{perf}_{\Z,\R}\cap \FLCA$ with a continuous action of a finite quotient $G$ of $G_F$. Then $\HH^q(B_{\hat{W}_F},A)$ is locally compact of finite ranks for $q=0,1,2$ and vanishes for all $q\ge 3$. Moreover, $\HH^2(B_{\hat{W}_F},A)$ is discrete of finite $\Q_p/\Z_p$-type.
\end{prop}
\begin{proof}
Let us consider the filtration of $A$ given by \cref{prop:degree0}. By \cref{rmk:structurervs,prop:structurecircle}, we have $\HH^q(B_{\hat{W}_F},A)\cong \HH^q(B_{\hat{W}_F},A_{\Z})$ for all $q\ge 2$. By \cref{thm:structurem}, this condensed abelian group is discrete of finite $\Q_p/\Z_p$-type for $q=2$, and vanishes for all $q\ge 3$.

Let us consider the morphism of long exact sequences in $\Ab(\CC)$  \[
 \begin{tikzcd}
0\ar[r] & (A_{\SSS^1})^G\arrow[equal]{d}\ar[r] & (F^1A)^G\arrow[equal]{d}\ar[r]& (A_{\R})^G\arrow[equal]{d}\ar[r]& \HH^1(B_G,A_{\SSS^1})\arrow[hook]{d}\\
0\ar[r] & \HH^0(B_{\hat{W}_F},A_{\SSS^1})\ar[r] & \HH^0(B_{\hat{W}_F},F^1A)\ar[r]& \HH^0(B_{\hat{W}_F},A_{\R})\ar[r,"\delta^0"]& \HH^1(B_{\hat{W}_F},A_{\SSS^1}).
\end{tikzcd}\]
The group $G$ is finite and $\underline{\Hom}(\R/\Z,A_{\SSS^1})$ is of finite $\Z$-type. Moreover we have \[\HH^1(B_G,A_{\SSS^1})=\HH^2(B_G,\underline{\Hom}(\R/\Z,A_{\SSS^1})).\] Hence $\HH^1(B_G,A_{\SSS^1})$ is finite and $\delta^0$ is the zero morphism. Since $\Ext^1_{\Ab(\CC)}(\HH^1(B_{\hat{W}_F},A_{\R}), \HH^1(B_{\hat{W}_F},A_{\SSS^1}))=0$, we have \[
\HH^1(B_{\hat{W}_F},F^1A)\cong  \HH^1(B_{\hat{W}_F},A_{\SSS^1})\oplus \HH^1(B_{\hat{W}_F},A_{\R}),
\] which is locally compact of finite ranks by \cref{rmk:structurervs,prop:structurecircle}. By the same argument, the morphism $\HH^0(B_{\hat{W}_F},A_{\Z})\to\HH^1(B_{\hat{W}_F},F^1A)$ has finite image, we call it $K$. Hence we have an exact sequence in $\Ab(\CC)$ \[
0\rightarrow \HH^1(B_{\hat{W}_F},F^1A)/K\rightarrow \HH^1(B_{\hat{W}_F},A)\rightarrow \HH^1(B_{\hat{W}_F},A_{\Z})\rightarrow 0,
\] where both the left and the right term are locally compact abelian groups of finite ranks. By \cref{prop:stableextensions}, $\HH^1(B_{\hat{W}_F},A)$ is a locally compact abelian group of finite ranks. 

We conclude by observing that by \cref{lowercohomologygroups},1), $\HH^0(B_{\hat{W}_F},A)=\HH^0(B_G,A)$ is represented by $A^G$ with subspace topology. This is a closed subgroup of $A$, hence it is locally compact of finite ranks.
\end{proof}

\begin{lem}\label{stabilitywrtcohomology}
Let $M\in \DDD^{perf}_{\Z,\R}$ such that $\H^q(M)$ is a locally compact abelian group for all $q$. Then $\H^q(M)\in \DDD^{perf}_{\Z,\R}\cap \FLCA$ for all $q$.
\end{lem}
\begin{proof}
Since $\H^q(M)$ is locally compact, the complex $R\underline{\Hom}(\R/\Z,\H^q(M))$ is concentrated in degrees $0,1$. For all $i$, we have an exact sequence in $\Ab(\CC)$\[
0\rightarrow \underline{\Ext}^1(\R/\Z,\H^{i-1}(M))\rightarrow \underline{\Ext}^i(\R/\Z,M)\rightarrow \underline{\Hom}(\R/\Z,\H^i(M))\rightarrow 0.
\] By hypothesis, $\underline{\Ext}^i(\R/\Z,M)$ is a finite $\Z$-type abelian group for all $i$. The exact sequence with $i=q$ implies that $\underline{\Hom}(\R/\Z,\H^q(M))$ is of finite $\Z$-type. The exact sequence with $i=q+1$ implies that $\underline{\Ext}^1(\R/\Z,\H^q(M))$ is of finite $\Z$-type. Thus $R\underline{\Hom}(\R/\Z,\H^q(M))\in \DDD^{perf}(\Z)$.
\end{proof}

\subsection{Duality for the cohomology of \texorpdfstring{$W_k$}{Wk}}\label{dualityforwk}
Let $R$ be a condensed commutative ring with 1 and $M$ a condensed $R[W_k]$-module. \cref{cohomologyz} gives \[
\HH^0(B_{W_k},M)=\mathrm{ker}(1-\varphi)=M^{W_k}, \qquad \HH^1(B_{W_k},M)=\mathrm{coker}(1-\varphi)=M_{W_k}, \qquad \HH^q(B_{W_k},M)=0 \quad \forall q\ge 2.
\]
\begin{cns}\label{cns:cupprodwk} Let $D$ be a condensed $R[W_k]$-module with the trivial action of $W_k$. Then $\HH^0(B_{W_k},D)=\HH^1(B_{W_k},D)=D$ and $\HH^q(B_{W_k},D)=0$ for all $q\ge 2$. Hence we have a trace map $R\GGamma(B_{W_k},D)\rightarrow D[-1]$. Let $M,N$ be condensed $R[W_k]$-modules. Suppose that there exists a $W_k$-equivariant perfect pairing of $R$-modules $M\times N \rightarrow D$. For all $q$, one gets an induced cup product pairing \begin{equation}\label{cupprodwk}
\HH^q(B_{W_k},M)\otimes_R \HH^{1-q}(B_{W_k},N) \rightarrow D.
\end{equation}
\end{cns}
\begin{thm}\label{wtfiniteduality}
Let $M,N,D$ as in \cref{cns:cupprodwk}. Suppose that $\underline{\Ext}^1_R(M_{W_k},D)=0$ and $\underline{\Ext}^1_R(N_{W_k},D)=0$. Then the cup-product pairing \eqref{cupprodwk} yields a perfect pairing for all $q$. 
\end{thm}
\begin{proof}
We consider the exact sequence of condensed $R$-modules \[
0\rightarrow M^{W_k}\rightarrow M \overset{1-\varphi}{\longrightarrow}M\rightarrow M_{W_k}\rightarrow 0.
\] Since $\underline{\Ext}^1_R(M_{W_k},D)=0$, the $W_k$-invariant isomorphism $\psi:N\rightarrow \underline{\Hom}_R(M,D)$ induces a morphism of exact sequences\[ \begin{tikzcd}
0\arrow[r] & \underline{\Hom}_R(M_{W_k},D)\arrow[r] &\underline{\Hom}_R(M,D)\arrow[r,"1-\varphi"] & \underline{\Hom}_R(M,D) \arrow[r] & \underline{\Hom}_R(M^{W_k},D)\rightarrow 0\\
0\arrow[r] & N^{W_k}\arrow[r]\arrow[u,"\psi^1"] & N\arrow[u,"\psi"]\arrow[r,"1-\varphi"] & N\arrow[u,"\psi"] \arrow[r] & N_{W_k}\arrow[u,"\psi^0"]\rightarrow 0.
\end{tikzcd}
\] One remarks that $\psi^1:\HH^0(B_{W_k},N)\to\underline{\Hom}_R(\HH^1(B_{W_k},M),D)$ and $\psi^0:\HH^1(B_{W_k},N)\to\underline{\Hom}_R(\HH^0(B_{W_k},M),D)$ are induced by \eqref{cupprodwk}. Since $\psi$ is an isomorphism, $\psi^0$ and $\psi^1$ are isomorphisms as well. Replacing $M$ with $N$, we conclude.
\end{proof}
\begin{ex}\label{pontrfinitefield}
Let $R=\Z$, $M$ be a $W_k$-module which is either discrete or compact Hausdorff as a condensed abelian group. Let $N=M^{\vee}$ and $D=\R/\Z$ with the trivial action. The perfect pairing $M\times M^{\vee} \rightarrow \R/\Z$ satisfies the hypotheses of \cref{wtfiniteduality}. Then for all $q$ we have a perfect pairing \[
\HH^q(B_{W_k},M)\otimes \HH^{1-q}(B_{W_k},M^{\vee})\rightarrow \R/\Z.
\]
\end{ex}
\begin{ex}\label{rlineardual}
Let $R=\R$ and $V$ a finite-dimensional real vector space with a continuous $\R$-linear action of $W_k$. The perfect pairing $V\times \underline{\Hom}_{\R}(V,\R) \rightarrow \R$ satisfies the hypotheses of \cref{wtfiniteduality}. For all $q$, we get a perfect pairing of condensed $\R$-modules \[
\HH^q(B_{W_k},V)\otimes_{\R}\HH^{1-q}(B_{W_k},V^*)\rightarrow \R.
\]
\end{ex}
\begin{ex}\label{rlinearlattice}
Let $R=\R$ and $V$ a finite-dimensional real vector space. Let $G$ be an extension of $W_k$ by a finite group $H$. Suppose that $V$ has a continuous $\R$-linear action of $G$. The perfect pairing $V\times \underline{\Hom}_{\R}(V,\R) \rightarrow \R$ induces a $\R$-linear and $W_k$-invariant pairing  $V^H\times (V^*)^H \rightarrow \R$. Since $H$ is finite, this pairing is perfect. 
Moreover, it satisfies the hypotheses of \cref{wtfiniteduality}. Hence for all $q$ we get a perfect cup-product pairing of condensed $\R$-modules \[
\HH^q(B_{W_k},V^H)\otimes_{\R} \HH^{1-q}(B_{W_k},(V^*)^H)\rightarrow \R.
\]
\end{ex}

\subsection{Conservativity of completion}\label{conservativityofcompletion}
\begin{defn}
Let $C$ be a complex of condensed abelian groups, and let $m$ be an integer. We define $C\otimes^L \Z/m$ as the cofiber of the morphism $\cdot m: C\to C$ in $\DDD(\Ab(\CC))$.
\end{defn}
\begin{ntt}
Let $A$ be a condensed abelian group and let $m\in \N$. We set $\left{m}{A}\coloneqq\mathrm{ker}(A\overset{\cdot m}{\rightarrow} A)$ and $A/m\coloneqq \mathrm{coker}(A\overset{\cdot m}{\rightarrow} A)$, where the kernel and the cokernel are computed in $\Ab(\CC)$. Moreover, we set \[ TA\coloneqq \underset{\substack{\leftarrow \\ m}}{\lim} \left{m}{A}, \qquad A^{nc}\coloneqq \underset{\substack{\leftarrow\\ m}}{\lim} \, A/m\] the \emph{Tate module} of $A$ and the \emph{naive completion} of $A$ respectively. Here the limits are computed in $\Ab(\CC)$. 
\end{ntt}
\begin{defn}[Completion of a complex of condensed abelian groups]\label{completionofcomplex}
Let $C$ be a complex of condensed abelian groups. We define its completion $C^{\wedge}$ as \[
C^{\wedge}\coloneqq R\lim_{\substack{\leftarrow \\ m}}\, (C\otimes^L \Z/m),
\] where the derived limit is computed in $\DDD(\Ab(\CC))$.
\end{defn}
We have an \emph{exact} functor \[
(-)^{\wedge}:\DDD(\Ab(\CC))\rightarrow \DDD(\Ab(\CC)),
\] and two functors \[
T:\Ab(\CC)\rightarrow \Ab(\CC),\qquad (-)^{nc}:\Ab(\CC)\rightarrow \Ab(\CC).
\]
In the topos of condensed sets the functor of $\N^{op}$-indexed limits has cohomological dimension $1$ \cite[Propositions 3.1.11,3.2.3]{proetale}. Hence for all $q$ we get a diagram with exact row and column \begin{equation}\label{cd1}
\begin{tikzcd}
&&&0\ar[d]&\\
&&&\H^q(C)^{nc}\ar[d]&\\
0\ar[r]& \underset{\leftarrow\, m}{\lim}^1 \left{m}{\H^q(C)}\ar[r]& \H^q(C^{\wedge})\ar[r]& \underset{\leftarrow\, m}{\lim} \H^q(C\otimes^L \Z/m)\ar[r]\ar[d] & 0.\\
&&&TH^{q+1}(C)\ar[d]&\\
&&&0&
\end{tikzcd}
\end{equation}
\begin{prop}\label{functoriality}
Let $a:A_1\to A_2$ be a morphism of condensed abelian groups. Suppose that we have \[
\underset{\leftarrow\, m}{\lim}^1 \, \left{m}{A_1}=\underset{\leftarrow\, m}{\lim}^1 \, \left{m}{A_2}=\underset{\leftarrow\, m}{\lim}^1 \, \left{m}{\mathrm{ker}(a)}=\underset{\leftarrow\, m}{\lim}^1 \, \left{m}{\mathrm{coker}(a)}=0.
\] Then the following properties hold: \begin{enumerate}[i)]
\item $\mathrm{ker}(Ta)=T\mathrm{ker}(a)$;
\item $\mathrm{coker}(a^{nc})=\mathrm{coker}(a)^{nc}$;
\item there exists a condensed abelian group which is both extension of $\mathrm{ker}(a^{nc})$ by $\mathrm{coker}(Ta)$ and extension of $T\mathrm{coker}(a)$ by $\mathrm{ker}(a)^{nc}$.
\end{enumerate}
\end{prop}
\begin{proof}
We set $\mathfrak{a}\coloneqq [A_1\overset{a}{\rightarrow} A_2]$, with $A_1$ in degree $-1$. By exactness of $(-)^{\wedge}$, we have fiber sequences in $\DDD(\Ab(\CC))$ \begin{equation}\label{eqn:cons6}
 A_1[0]^{\wedge}\rightarrow A_2[0]^{\wedge} \rightarrow \mathfrak{a}^{\wedge}
 \end{equation} and
 \begin{equation}\label{eqn:cons7}
 \mathrm{ker}(a)[1]^{\wedge}\rightarrow \mathfrak{a}^{\wedge}\rightarrow \mathrm{coker}(a)[0]^{\wedge}.
 \end{equation}
By the hypothesis and by \eqref{cd1}, $A_1[0]^{\wedge}$ is concentrated in degrees $-1,0$. Moreover, $\H^{-1}(A_1[0]^{\wedge})=TA$ and $\H^0(A_1[0]^{\wedge})=A_1^{nc}$. The same holds for $A_2,\mathrm{ker}(a),\mathrm{coker}(a)$.
The long exact sequences coming from \eqref{eqn:cons6} and \eqref{eqn:cons7} conclude the proof. 
\end{proof}
\begin{prop}[Conservativity of completion]\label{conservativitycompletion}
Let $C,D$ be complexes of condensed abelian groups cohomologically concentrated in degrees $0,1,2$. 
Let $f:C\rightarrow D$ be a map in $\DDD(\Ab(\CC))$ such that $f^{\wedge}:C^{\wedge}\to D^{\wedge}$ is an equivalence. Suppose one of the two (dual) conditions are satisfied \begin{enumerate}[(i)]
\item $\H^0(C),\H^1(C),\H^0(D),\H^1(D)$ are of finite $\Z$-type, and $\H^2(C)$ and $\H^2(D)$ are of finite $\Q_p/\Z_p$-type.
\item $\H^0(C)^{\vee},\H^0(D)^{\vee}$ are of finite $\Q_p/\Z_p$-type, and $\H^1(C)^{\vee}$, $\H^2(C)^{\vee}$, $\H^1(D)^{\vee}$, $\H^2(D)^{\vee}$ are of finite $\Z$-type. 
\end{enumerate}
Then $f$ is an equivalence.
\end{prop}
\begin{rmk}\label{rmk:lim1}
If $A$ is a condensed abelian group, and either $A$ or $A^{\vee}$ is of one of the two types present in the hypotheses (i) and (ii), then $\left{m}{A}$ is finite for all $m$. Hence the systems $\{\left{m}{A}\}$ are Mittag-Leffler and we have \[
\underset{\leftarrow\, m}{\lim}^1 \, \left{m}{A}=0.
\]
\end{rmk}
\begin{proof}
The map $f:C\rightarrow D$ in $\DDD(\Ab(\CC))$ induces $f^q:\H^q(C)\to\H^q(D)$ in $\Ab(\CC)$. We need to prove that $\mathrm{ker}(f^q)=\mathrm{coker}(f^q)=0$ for all $q$. We suppose that hypothesis $(i)$ holds. By \cref{extsubquo}, (\ref{extsubquo:sub}) and (\ref{extsubquo:quo}), if $\H^q(C)$ and $\H^q(D)$ are both of finite $\Z$-type (resp.\ of finite $\Q_p/\Z_p$-type), then so are $\mathrm{ker}(f^q)$ and $\mathrm{coker}(f^q)$. Thus there exist  integers $r_0,r_1,s_0,s_1,\rho,\sigma$ and finite groups $F_0,F_1,G_0,G_1,H,K$ such that we have \[
\begin{split}
\mathrm{ker}(f^q)=\Z^{r_q}\oplus F_q, \qquad \qquad  &\mathrm{coker}(f^q)=\Z^{s_q}\oplus G_q, \qquad q=0,1,\\
\mathrm{ker}(f^2)=(\Q_p/\Z_p)^{\rho}\oplus H, \qquad \qquad & \mathrm{coker}(f^2)=(\Q_p/\Z_p)^{\sigma}\oplus K.
\end{split}
\]
Applying \cref{rmk:lim1} to the diagram \eqref{cd1}, we get morphisms of exact sequences for all $q$ \[
\begin{tikzcd}
0\arrow[r]& \H^q(C)^{nc}\arrow[r]\arrow[d,"(f^q)^{nc}"] & \H^q(C^{\wedge})\arrow[r]\arrow[d,"(f^{\wedge})^q"] & T\H^{q+1}(C)\arrow[r]\arrow[d,"Tf^{q+1}"]& 0\\
0\arrow[r]& \H^q(D)^{nc}\arrow[r]& \H^q(D^{\wedge})\arrow[r]& T\H^{q+1}(D)\arrow[r]& 0.
\end{tikzcd}
\] By hypothesis, $(f^{\wedge})^q$ is an isomorphism. The Snake Lemma and \cref{functoriality} give us \begin{enumerate}[A)]
\item $\mathrm{ker}(f^q)^{nc}=0$ for all $q$;
\item $T\mathrm{coker}(f^q)=0$ for all $q$;
\item $T\mathrm{ker}(f^{q+1})=\mathrm{coker}(f^q)^{nc}$ for all $q$.
\end{enumerate}
A) implies that $r_0=r_1=0$ and $F_0=F_1=H=0$, and B) implies that $\sigma=0$. Consequently, we have
\[
\begin{split}
\mathrm{ker}(f^q)=0, \qquad \qquad  &\mathrm{coker}(f^q)=\Z^{s_q}\oplus G_q, \qquad q=0,1,\\
\mathrm{ker}(f^2)=(\Q_p/\Z_p)^{\rho}, \qquad \qquad & \mathrm{coker}(f^2)=K.
\end{split}
\]
We now apply C), hence we have \[
\begin{split}
T\mathrm{ker}(f^1)=\mathrm{coker}(f^0)^{nc} &\implies 0=\hat{\Z}^{s_0}\oplus G_0 \implies s_0=0,\, G_0=0\\
T\mathrm{ker}(f^2)=\mathrm{coker}(f^1)^{nc} &\implies \Z_p^{\rho}=\hat{\Z}^{s_1}\oplus G_1 \implies s_1=\rho=0, \, G_1=0\\
T\mathrm{ker}(f^3)=\mathrm{coker}(f^2)^{nc} & \implies 0=K.
\end{split}
\] 
Thus $f$ is an equivalence.

The proof with hypothesis (ii) can be done in the same way. \end{proof}

\subsection{The main result}\label{mainresult}
\begin{thm}\label{thm:dualityrealvs}
Let $M\in \DDD^{perf}(\R)$ with a continuous action of $W_F/U$, with $U\subset I$ open normal subgroup. Then one has $R\GGamma(B_{\hat{W}_F},M),R\GGamma(B_{\hat{W}_F},M^D)\in\DDD^b(\FLCA)$. Moreover, the trace map \[
R\GGamma(B_{\hat{W}_F},\R/\Z(1))\rightarrow \R/\Z[-2]
\] induces a perfect cup-product pairing \[ R\GGamma(B_{\hat{W}_F},M)\otimes^L R\GGamma(B_{\hat{W}_F},M^D)\rightarrow \R/\Z[-2] \]
in $\DDD^b(\Ab(\CC))$.
\end{thm}
\begin{rmk}
A fortiori, the result holds if $M\in \DDD^{perf}(\R)$ has a continuous action of a finite quotient $G$ of $G_F$.
\end{rmk}
\begin{proof}
By induction on the length of the complex, we suppose that $M=V$ is a finite-dimensional real vector space with a continuous action of $W_F/U$. By \cref{rmk:structurervs}, we have $R\GGamma(B_{\hat{W}_F},V),R\GGamma(B_{\hat{W}_F},V^D)\in \DDD^b(\FLCA)$. 

The morphism $\R(1)\to \R/\Z(1)$ induces a morphism of cup-product pairings 
\[
\begin{tikzcd}
R\GGamma(B_{\hat{W}_F},V)\otimes^L R\GGamma(B_{\hat{W}_F},V^*)[-1]\arrow[r]\arrow[d]&R\GGamma(B_{\hat{W}_F},V)\otimes^L R\GGamma(B_{\hat{W}_F},V^D)\arrow[d]\\
\R[-2]\arrow[r]& \R/\Z[-2].
\end{tikzcd}
\] By \cref{rmk:dualofv}, $V^*[-1]\to V^D$ is an equivalence. Since $R\GGamma(B_{\hat{W}_F},V)$ and $R\GGamma(B_{\hat{W}_F},V^*)$ are complexes of $\DDD^{perf}(\R)$ (see \cref{rmk:structurervs}), we have \[ R\underline{\Hom}(R\GGamma(B_{\hat{W}_F},V),\R)=R\underline{\Hom}(R\GGamma(B_{\hat{W}_F},V),\R/\Z)\] and the same for $V^*$. Moreover, we have \[R\underline{\Hom}(R\GGamma(B_{\hat{W}_F},V),\R[-1])=R\underline{\Hom}_{\R}(R\GGamma(B_{\hat{W}_F},V),\R[-1])\] and the same for $V^*$. Consequently, it is enough to prove that the cup-product pairing of condensed $\R$-modules \[
R\GGamma(B_{\hat{W}_F},V)\otimes^L_{\R}R\GGamma(B_{\hat{W}_F},V^*)\rightarrow \R[-1]
\] is perfect. This coincides with the cup-product pairing \[
R\GGamma(B_{W_k},V^{I/U})\otimes_{\R}^L R\GGamma(B_{W_k},(V^*)^{I/U})\rightarrow \R[-1].
\] We conclude as in \cref{rlinearlattice}.
\end{proof}

\begin{prop}[Condensed Weil-Tate Local Duality]\label{condensedweiltate}
Let $M$ be a finite abelian group with a continuous action of $W_F$. Then the trace map \[
R\GGamma(B_{\hat{W}_F},\R/\Z(1))\rightarrow \R/\Z[-2]
\] induces a perfect cup-product pairing \[
R\GGamma(B_{\hat{W}_F},M)\otimes^L R\GGamma(B_{\hat{W}_F},M^D)\longrightarrow \R/\Z.
\]
\end{prop}
\begin{proof}
Let $\mu\subset \overline{L}^{\times}$ denote the discrete continuous $W_F$-module of roots of unity. We set $M'\coloneqq \Hom(M,\mu)$. By \cref{dualizingobjectfinite,rmk:cohomdiscretemod}, the cup-product pairing becomes \[
\underline{R\Gamma(B_{W_F}(\Set),M)}\otimes^L \underline{R\Gamma(B_{W_F}(\Set),M')}\rightarrow \R/\Z.
\] By \cref{finitecoeffstructure}, we have $R\underline{\Hom}(\underline{R\Gamma(B_{W_F}(\Set),M)},\R/\Z)=\underline{R\Hom(R\Gamma(B_{W_F}(\Set),M),\Q/\Z)}$ and the same for $M'$. Hence it is enough to prove that the cup-product pairing in $\DDD^{perf}(\Z)$ \[
R\Gamma(B_{W_F}(\Set),M)\otimes^L R\Gamma(B_{W_F}(\Set),M') \longrightarrow \Q/\Z
\] is perfect. This follows from 
\cite[Proposition 4.1.1]{Karpuk} and \cite[Theorem 2.1]{ADT}.
\end{proof}
\begin{prop}\label{thm:dualityfgcoefficients}
Let $M\in \DDD^{perf}(\Z)$ with a continuous action of a finite quotient $G$ of $G_F$. Then one has $R\GGamma(B_{\hat{W}_F},M),R\GGamma(B_{\hat{W}_F},M^D)\in \DDD^b(\FLCA)$. Moreover, the trace map \[
R\GGamma(B_{\hat{W}_F},\R/\Z(1))\rightarrow \R/\Z[-2]
\] induces a perfect cup-product pairing \[
R\GGamma(B_{\hat{W}_F},M)\otimes^L R\GGamma(B_{\hat{W}_F},M^D)\rightarrow \R/\Z[-2]
\] in $\DDD^b(\Ab(\CC))$.
\end{prop}
\begin{proof}
By induction on the length of the complex and by \cref{condensedweiltate}, we can suppose that $M$ is a free abelian group of finite $\Z$-type. By \cref{thm:structurem,thm:structuremd}, we have $R\GGamma(B_{\hat{W}_F},M),R\GGamma(B_{\hat{W}_F},M^D)\in \DDD^b(\FLCA)$.

Let $\psi(M^D)$ and $\psi(M)$ be defined as in \eqref{cupprodmap}. Let us consider $\psi(M^D)^{\wedge}=R\lim_m \psi((M\otimes^L \Z/m)^D)$ and $\psi(M)^{\wedge}=R\lim_m \psi(M\otimes^L \Z/m).$ Localising at $EG$, we have $M\otimes^L \Z/m\cong M/m$ in $\DDD^b(B_{\hat{W}_F})$. Hence $M\otimes^L \Z/m$ is a finite abelian group with a continuous action of $G$. By \cref{condensedweiltate}, $\psi((M/m)^D)$ and $\psi(M/m)$ are equivalences. Hence so are $\psi(M^D)^{\wedge}$ and $\psi(M)^{\wedge}$.
By \cref{thm:structurem} (resp.\ \cref{thm:structuremd}), $\psi(M^D)$ (resp.\ $\psi(M)$) satisfies the hypothesis (i) (resp.\ (ii)) of \cref{conservativitycompletion}. Thus $\psi(M^D)$ (resp.\ $\psi(M)$) is an equivalence. This concludes the proof.
\end{proof}
\begin{thm}\label{thm:dualitythm}
Let $M\in\DDD^{perf}_{\Z,\R}$ with a continuous action of a finite quotient $G$ of $G_F$. Then one has $R\GGamma(B_{\hat{W}_F},M),R\GGamma(B_{\hat{W}_F},M^D)\in \DDD^b(\FLCA)$. Moreover, the trace map \[
R\GGamma(B_{\hat{W}_F},\R/\Z(1))\rightarrow \R/\Z[-2]
\] induces a perfect cup-product pairing \[
R\GGamma(B_{\hat{W}_F},M)\otimes^L R\GGamma(B_{\hat{W}_F},M^D)\rightarrow \R/\Z[-2]
\] in $\DDD^b(\Ab(\CC))$. 

If $\H^q(M)$ is a locally compact abelian group for all $q$, the induced cup-product pairing \[
\HH^q(B_{\hat{W}_F},M)\otimes \HH^{2-q}(B_{\hat{W}_F},M^D)\rightarrow \HH^2(B_{\hat{W}_F},\R/\Z(1))=\R/\Z
\] is a perfect pairing of locally compact abelian groups of finite ranks.
\end{thm}

\begin{proof}
We have a distinguished triangle in $\DDD(B_{\hat{W}_F})$ \[
R\underline{\Hom}(\R/\Z,M)\rightarrow R\underline{\Hom}(\R,M)\rightarrow M.
\] Since $M\in \DDD^{perf}_{\Z,\R}$, we have $R\underline{\Hom}(\R/\Z,M)\in \DDD^{perf}(\Z)$ and $R\underline{\Hom}(\R,M)\in \DDD^{perf}(\R)$. By \cref{thm:dualityfgcoefficients,thm:dualityrealvs}, the fist two statements of the theorem hold for $R\underline{\Hom}(\R/\Z,M)$ and $R\underline{\Hom}(\R,M)$ instead of $M$. 
Hence the same holds for $M$.

Let us suppose that $\H^q(M)$ is locally compact for all $q$. We have the spectral sequence
\[
E_2^{i,j}=\HH^i(B_{\hat{W}_F},\H^j(M))\Longrightarrow \HH^{i+j}(B_{\hat{W}_F},M).
\] By \cref{stabilitywrtcohomology}, we have $\H^j(M)\in \DDD^{perf}_{\R,\Z}\cap \FLCA$ for all $j$. By \cref{prop:degree0cohomology}, $E_2^{i,j}$ is locally compact of finite ranks for all $i,j$, and vanishes for all $i\ge 3$. Moreover, $E_2^{2,j}$ is discrete for all $j$. Thus the morphism $E_2^{0,j+1}\rightarrow E_2^{2,j}$ is strict for all $j$. These are the only morphisms of the second page. Moreover, for all $n\ge 3$ and for all $i,j$, we have $d_n^{i,j}=0$. Consequently, the condensed abelian group $E_{\infty}^{i,j}$ is a locally compact abelian group of finite ranks for all $i,j$. By \cref{prop:stableextensions}, we have $\HH^q(B_{\hat{W}_F},M)\in\FLCA$ for all $q$. Consequently, we have $\underline{\Ext}^i(\HH^q(B_{\hat{W}_F},M),\R/\Z)=0$ for all $i\ge 1$ and for all $q$. Hence the equivalence \[
R\GGamma(B_{\hat{W}_F},M^D)\cong R\underline{\Hom}(R\GGamma(B_{\hat{W}_F},M),\R/\Z[-2])
\] yields \[\HH^q(B_{\hat{W}_F},M^D)\cong \H^q(R\underline{\Hom}(R\GGamma(B_{\hat{W}_F},M),\R/\Z[-2]))\cong \HH^{2-q}(B_{\hat{W}_F},M)^{\vee},\] which is locally compact of finite ranks. The result follows. 
\end{proof}
\begin{rmk}
By \cref{internalrhom}, the corresponding cup-product pairing in $\DDD^b(\FLCA)$ is perfect too.
\end{rmk}

\begin{cor}
Let $M$ be a discrete countable abelian group. Then for all $q$ the condensed abelian group $\HH^q(B_{\hat{W}_F},M)$ is represented by a $\kappa$-small discrete abelian group and $\HH^q(B_{\hat{W}_F},M^D)$ is represented by a $\kappa$-small compact Hausdorff topological abelian group. Moreover, the trace map \[
R\GGamma(B_{\hat{W}_F},\R/\Z(1))\rightarrow \R/\Z[-2]
\] induces a perfect cup-product pairing \[
\HH^q(B_{\hat{W}_F},M)\otimes \HH^{2-q}(B_{\hat{W}_F},M^D)\rightarrow \HH^2(B_{\hat{W}_F},\R/\Z(1))=\R/\Z
\] of $\kappa$-small locally compact abelian groups.
\end{cor}
\begin{proof}
Firstly, we observe that we can write \[
M=\underset{\substack{\rightarrow\\ n}}{\lim}\,  M_n,
\] where $M_n$ is a $G$-equivariant finitely generated abelian group. In order to show this, we observe that since $M$ is countable, we can write $M=\{x_i\}_{i\in \N}$. Hence, for all $n\in \N$, we consider $M_n$ as the subgroup of $M$ generated by $\{g\cdot x_i\}_{g\in G, 0\le i\le n}$. Then $M_n$ is $G$-equivariant and finitely generated for all $n$, and if $n_1<n_2$, we have $M_{n_1}\subset M_{n_2}$. 

Consequently, we have \[
\HH^q(B_{\hat{W}_F},M)=\underset{\substack{\rightarrow\\ n}}{\lim} \,\HH^q(B_{\hat{W}_F},M_n),
\] which is $\kappa$-small and discrete for all $q$. Moreover, since we have $M^D=R\underset{\substack{\leftarrow \\ n}}{\lim}\, M_n^D$, there is an exact sequence in $\Ab(\CC)$ \[
0\rightarrow \underset{\substack{\leftarrow\\n}}{\lim}^1 \, \HH^{q-1}(B_{\hat{W}_F},M_n^D)\rightarrow \HH^q(B_{\hat{W}_F},M^D)\rightarrow \underset{\substack{\leftarrow\\n}}{\lim} \, \HH^{q}(B_{\hat{W}_F},M_n^D)\rightarrow 0
\] for all $q$. By \cref{thm:structuremd}, the condensed abelian groups $\HH^{q-1}(B_{\hat{W}_F},M_n^D)$ and $\HH^{q}(B_{\hat{W}_F},M_n^D)$ are compact Hausdorff of finite ranks for all $n$. Hence $\HH^q(B_{\hat{W}_F},M)$ is $\kappa$-small compact Hausdorff for all $q$ by \cite[Proposition 3.2]{SolidGC}. In particular, since $\DDD^b(\LCA_{\kappa})\rightarrow \DDD^b(\Ab(\CC))$ is a stable $\infty$-subcategory, we have $R\GGamma(B_{\hat{W}_F},M),R\GGamma(B_{\hat{W}_F},M^D)\in \DDD^b(\LCA_{\kappa})$. Moreover, the cup-product pairing \[
R\GGamma(B_{\hat{W}_F},M)\otimes^L R\GGamma(B_{\hat{W}_F},M^D)\rightarrow \R/\Z[-2]
\] induces the morphisms \[
\psi(M^D): R\GGamma(B_{\hat{W}_F},M)\rightarrow R\underline{\Hom}(R\GGamma(B_{\hat{W}_F},M^D),\R/\Z[-2])
\] and \[
\psi(M): R\GGamma(B_{\hat{W}_F},M^D)\rightarrow R\underline{\Hom}(R\GGamma(B_{\hat{W}_F},M),\R/\Z[-2]).
\] It is easy to see that we have $\psi(M)=R\underset{\substack{\leftarrow\\n}}{\lim} \, \psi(M_n)$. Moreover, since $R\underline{\Hom}(-,\R/\Z[-2]):\DDD^b(\LCA_{\kappa})^{op}\rightarrow \DDD^b(\LCA_{\kappa})$ is an equivalence, we also have $\psi(M^D)=\underset{\substack{\rightarrow\\n}}{\lim} \,\psi(M_n^D)$. Consequently, both $\psi(M)$ and $\psi(M^D)$ are equivalences by \cref{thm:dualitythm}. Moreover, since $\underline{\Ext}^i(\HH^q(B_{\hat{W}_F},M),\R/\Z)=0$ for all $i\ge 1$ and for all $q$, taking $q$th cohomology on $\psi(M)$ yields \[
\HH^q(B_{\hat{W}_F},M^D)\cong \HH^{2-q}(B_{\hat{W}_F},M)^{\vee}
\] and similarly for $\psi(M^D)$. The result follows.
\end{proof}

\begin{ex}
Let us consider $M=\R/\Z$. Then we have \[ \HH^1(B_{\hat{W}_F},\R/\Z)=\underset{\substack{\rightarrow\\ U\subset I \\ \text{open normal}\\\text{subgr.}}}{\lim} \, \HH^1(B_{W_F/U},\R/\Z)= \underset{\substack{\rightarrow\\ U\subset I \\ \text{open normal}\\ \text{subgr.}}}{\lim} \, ((W_F/U)^{ab})^{\vee}=(W_F^{ab})^{\vee}.\] Moreover, one gets \[\HH^1(B_{\hat{W}_F},M^D)=\HH^1(B_{\hat{W}_F},\Z(1))=F^{\times}.\]
Hence the perfect cup product pairing of \cref{thm:dualitythm} yields the isomorphism of topological abelian groups \[
F^{\times}\overset{\sim}{\rightarrow} W_F^{ab},
\] which comes from Local Class Field Theory ``à la Weil".
\end{ex}

\newpage 
\printbibliography
\end{document}